\documentclass[a4paper, 11pt]{amsart}
\usepackage {a4wide, epsfig, enumerate}

\usepackage[utf8]{inputenc}  
\usepackage{lmodern}
\usepackage{amsfonts}
\usepackage{amsmath}
\usepackage{amssymb}
\usepackage{amsthm}
\usepackage{fontenc}
\usepackage{graphicx}
\usepackage{multido}         
\usepackage{pst-all}
\usepackage{pstricks}        
\usepackage{textcomp}     
\usepackage{dsfont}
\usepackage[all]{xy}
\usepackage{mathrsfs}        

\usepackage{pstricks-add} 
\usepackage{graphicx}     
\usepackage{psfrag}
\usepackage{cite}
\usepackage{multicol}       

\usepackage{pstricks-add} 
\usepackage[hang,small]{caption} 
\usepackage[]{enumitem}

\newcommand{\C}{{\mathbb C}}
\newcommand{\R}{{\mathbb R}}
\newcommand{\N}{{\mathbb N}}

\newcommand{\Q}{{\mathbb Q}}

\newcommand{\proj}{{\mathbb{P}}}

\newcommand{\Cscr}{\mathscr{C}}
\newcommand{\Cr}{\mathcal{C}}
\newcommand{\Cc}{{\mathcal C}}

\newcommand{\Qr}{\mathcal{Q}}
\newcommand{\D}{\text{D}}

\newcommand{\dist}{\text{dist}}
\newcommand{\Exp}{Exp}
\newcommand{\B}{\mathbb{B}}

\newcommand{\A}{\mathscr{A}}
\newcommand{\Fr}{\mathscr{F}}
\newcommand{\m}{\mathrm{m}}
\newcommand{\hm}{\hat{\m}}

\newcommand{\Fc}{\mathcal{F}}

\newcommand{\Rr}{\mathscr{R}}
\newcommand{\Br}{\mathscr{B}}
\newcommand{\Dr}{\mathscr{D}}

\newcommand{\Wr}{\mathscr{W}}

\newcommand{\hx}{{\hat{x}}}
\newcommand{\hy}{{\hat{y}}}
\newcommand{\hp}{{\hat{p}}}
\newcommand{\hA}{{\hat{A}}}
\newcommand{\hP}{{\hat{P}}}
\newcommand{\hL}{{\hat{L}}}
\newcommand{\hX}{{\hat{X}}}
\newcommand{\hq}{{\hat{q}}}
\newcommand{\hf}{{\hat{f}}}

\newcommand{\hW}{{\hat{W}}}
\newcommand{\hR}{{\hat{\Rr}}}

\newcommand{\hY}{{\hat{Y}}}
\newcommand{\hnu}{{\hat{\nu}}}

\newcommand{\hxi}{\hat{\xi}}

\newcommand{\intgeom}{\dot{\wedge}}

\newcommand{\diam}{\mathrm{diam}}
\newcommand{\supp}{\mathrm{supp}}
\newcommand{\card}{\mathrm{card}}
\newcommand{\id}{\mathrm{id}}
\newcommand{\dd}{\mathrm{dd}}

\newcommand{\ds}{\mathrm{d}}

\newcommand{\genus}{\mathrm{genus}}

\newcommand{\hol}{\mathrm{hol}}
\newcommand{\Jac}{\mathrm{Jac}}

\theoremstyle{plain}
\newtheorem{theoreme}{Theorem}[section]
\newtheorem{lemma}[theoreme]{Lemma}

\newtheorem{lemme}[theoreme]{Lemma}
\newtheorem{prop}[theoreme]{Proposition}
\newtheorem{coro}[theoreme]{Corollary}

\newtheorem{defi}[theoreme]{Definition}
\newtheorem{defif}[theoreme]{Definition}
\newtheorem{prop/defi}[theoreme]{Proposition/Definition}

\theoremstyle{definition}

\newtheorem*{question}{Question}

\theoremstyle{remark}
\newtheorem*{rmq}{Remark}
\newtheorem*{rmk}{Remark}

\usepackage[T1]{fontenc}
 \usepackage{authblk}
 \usepackage{amsaddr}
\author[Sandrine Daurat]{Sandrine Daurat }
\begin{document}
\title[Hyperbolic saddle measures and laminarity]{Hyperbolic saddle measures and laminarity for holomorphic endomorphisms of  $\proj^2\C$}

\maketitle
\begin{center}
\today
\end{center}

\renewcommand{\abstractname}{Abstract}
\begin{abstract}
We study the laminarity of the Green current of endomorphisms of $\proj^2(\C)$ near hyperbolic measures of saddle type.
When these measures are supported by attracting sets, we prove that the Green current is laminar in the basin of attraction and we obtain new ergodic properties.
This generalizes some results of Bedford and Jonsson on regular polynomial mappings in
 $\C^2$.
\end{abstract}

\section{Introduction}\label{section intro laminarite}

This article concerns the dynamics of a holomorphic endomorphism $f$ of $\proj^2(\C)$ (hereafter denoted $\proj^2$).
Recall that the Julia set $J_1$ is the complementary of the regular part of the dynamics and $J_1=\supp(T)$, where $T$ is the Green current of $f$
The most chaotic part of the dynamics is a subset $J_2$ of $J_1$ that corresponds to the support of the equilibrium measure $\mu_{eq}=T\wedge T$.
See \cite{din:sib:survey}, and the references therein, for more results about the dynamics on $J_2$ and the proprieties of $\mu_{eq}$.
The natural measure to consider on $J_1$ is the trace measure $\sigma_T=T\wedge \omega_{FS}$ of the Green current $T$ which is invariant.
 We address the question of understanding the behavior  of $\sigma_T$-almost every point when  $J_2=\supp(\mu_{eq})\neq J_1$.

\bigskip

We are mainly interested in the case where $f$  admits a trapping region   $ U $, i.e. an open set such that $f(U)\Subset U$.
The decreasing limit $\A=\bigcap_{n\in \N}f^n(U)$ is called an attracting set.
 Notice that, since $ U\neq \proj^2$ is a trapping region, $\supp(\mu_{eq})\cap U=\emptyset$. 
 If $U$ is Kobayashi hyperbolic then $\A$ is a finite union of attracting periodic orbits, and the dynamics in the basin of $\A$ is well understood. 
So assume that
  $ U\neq \proj^2 $ and  $U$ is not Kobayashi hyperbolic.
In this case, $U$ contains a curve $\ell$, see Proposition \ref{prop A non trivial}.

Bedford-Jonsson \cite{bed:jon} considered the case of endomorphisms
of the form $$f_0:[x:y:z]\mapsto [P(x,y,z):Q(x,y,z):z^d],$$
near the line at infinity $L_\infty=\{z=0\}$, which is  an attracting set.
 From the dynamics in $\proj^1$, it is known that there exists a unique  measure $\mu_\infty$ 
 of maximal entropy on $L_\infty$, and $\mu_\infty$ represents the equidistribution of saddle points in $L_\infty$. Moreover, the disintegration of the measure $\mu_\infty$ on the unstable manifold $L_\infty$ is induced by $T$, i.e. $\mu_\infty=T\wedge [L_\infty]$.
 
In \cite{bed:jon}, Bedford-Jonsson prove that the Green current of $f_0$ is laminar subordinate to the stable manifolds of $\mu_\infty$ in the basin of $L_\infty$. Thereby, they also obtain that $\mu_\infty $ represent the equidistribution of $\sigma_T$-almost every points  in the basin of $L_\infty$, see Definition \ref{def laminaire}.

\bigskip
 
In general, attracting sets have a more complicated structure. In particular, they are generically non-algebraic, see \cite{jon:wei,moi:taf} and the references therein. 
However, we are going to see that the dynamics in the basin of attraction is similar to the case describe above. For technical reasons, we assume :

\begin{enumerate}[label={$(Tub)$}]
\item \label{cond U tub} 
$U$ is a tubular neighbourhood of a curve $\ell$
\end{enumerate}
In particular, $U$  is a euclidean retract\footnote{See \cite[Proposition/Definition IV 8.5]{dol}}  of  $\ell$.
We also assume that $f$ satisfies one of the following : 
\begin{enumerate}[label={$(Sd_t)$}]
\item \label{cond petit dt} 
 $\A$ is an attracting set of small topological degree,
\end{enumerate}
or
\begin{enumerate}[label={$(SJ)$}]
\item \label{cond small jacobian} 
There exists a neighbourhood $N$ of $\A$ in which  the Jacobian of $f$ is small, i.e. there exists $0<\alpha<1$ such that for all $p=[x:y:z]\in N$, $|\Jac_p(f)|<\alpha \max(|x|,|y|,|z|)^{2d-2}$.
\end{enumerate}

The condition  \ref{cond small jacobian} is typically satisfied by small perturbations of $f_0$, see also section \ref{section rmk}.
 We refer Definition \ref{def small dt} for the definition of small topological degree attracting sets and to \cite{moi} and \cite{moi:taf} examples.
Under these assumptions, we know

\begin{theoreme}[\cite{DINH,moi,moi:taf}]\label{th dinh,da-ta}
Let $f$ be an endomorphism of $\proj^2$ admitting a trapping region $U$. Assume that $f$ and $U$ satisfy the conditions  \ref{cond U tub}, and \ref{cond petit dt} or \ref{cond small jacobian}. Then

\begin{enumerate}[label={$(C_\arabic*)$}]
\item \label{ccl unicite} 
$ \text{there exists a unique invariant current } T^u\in \Cc_{(1,1)}(U),
$ where $\Cc_{(1,1)}(U)$ is the set of positive closed currents of bidegree $(1,1)$ with supports in $U$,
\item \label{ccl forte on nu} 
$\nu=T\wedge T^u \text{ is mixing, of entropy }\log(d), $
\item \label{ccl all hyp}
all measure of entropy $\log(d) $ and support on $\A$ is hyperbolic of saddle type.
\end{enumerate}
\end{theoreme}

One of the main result of this article is to prove that $\nu$ has the same properties as the $\mu_\infty$ in \cite{bed:jon}.
 
\begin{theoreme}\label{th intro}
Let $f$ be an endomorphism of $\proj^2$  admitting a trapping region $U$. 
Assume that $f$ and $U$ satisfy  \ref{cond U tub}, and \ref{cond petit dt} or \ref{cond small jacobian}.

Then the following is true:
\begin{enumerate}[label={(\alph*)}]
\item \label{unicite nu} $\nu=T\wedge T^u$ is the unique measure of maximal entropy $\log(d)$ in $\Br_\A=\bigcup_{n\in \N} f^{-n}(U)$,
\item \label{eq per pts} if $Per_n$ denote the set of periodic points of period $n$ then 
$$ \nu_n:=\dfrac{1}{d^n}\sum_{\kappa \in Per_n\cap \Br_\A} \delta_\kappa \rightarrow \nu, \text{  as } n\rightarrow\infty$$
\item \label{conditionals}  the conditionals of $\nu$ on unstable manifolds are induced by $T$.
\end{enumerate}
\end{theoreme}

We are going to prove each point of Theorem \ref{th intro} separately under  weaker assumptions, see section \ref{section assumptions}.
In \cite{moi:taf,DINH}, the curve $\ell$ is a line but it can also be a conic as in \cite[section 5]{for:wei}. See also Section \ref{section hyp th lam bassin}, where we extend the results of \cite{DINH} to this setting.

\bigskip

With the conclusions \ref{ccl unicite} and \ref{ccl forte on nu} of Theorem \ref{th dinh,da-ta},  it is quite elementary to prove that there  exists an open neighboorhood $W$ of $\A$ such that $(f^n)_*(\sigma_T|_W)\rightharpoonup\nu$ for the weak-$*$ topology.
Nevertheless, this does not provide, a priori, a description of the dynamics of $\sigma_T$-a.e. point $p\in W$.
To this end, we are going to establish the laminarity of the Green current in the basin of attraction.

This question has been studied by several autheurs, see for exemple \cite{deT:laminarite,deT:concentration} and  \cite{duj:fatou}.
  Dujardin \cite{duj:contre} constructed  examples of skew-products of $\C^2$, that can be extended as endomorphisms of 
$\proj^2$, for which the  Green current is not laminar near an invariant fibre $F$ that is not attracting.

In  \cite{for:sib:hyp}, the authors established the laminarity of the Green current in the neighboorhood of a uniformly hyperbolic saddle set. We obtain the following result in the non-uniformly hyperbolic case. This generalises  \cite{for:sib:hyp} and \cite{bed:jon} to the basin of an attracting set and, more generally, to the basin of a hyperbolic measure of saddle type.

\begin{theoreme}\label{th laminarite sans ens. att.}
Let $f$ be an endomorphism of $\proj^2$ of degree $d$ and $T$ be its  Green current.
Assume that there exists an invariant current  $T^u$ ($\frac{1}{d}f_*T^u=T^u$) such that
$
\text{the measure } \nu=T\wedge T^u \text{ is ergodic,  of entropy }\log(d), \text{  hyperbolic of saddle type,}$
and $\supp(\nu)\cap \supp(\mu_{eq})=\emptyset$, where $\mu_{eq}=T\wedge T$ is the equilibrium measure. 

Then there  exists a non trivial positive current $T^s$ of bidegree $(1,1)$ which is  laminar and  subordinate to the stable manifolds $\bigcup_{x\in \supp(\nu)} W^s(x)$ such that $T^s\leq T$
and for $\nu$-a.e. $x\in \A$,  $W^s(x)\subset \supp(T^s)$. 
\end{theoreme}

Among other things, Theorem \ref{th laminarite sans ens. att.} gives us information on the topological structure of the Julia set $J_1=\supp (T)$. We deduce the following:

\begin{coro}
Let $\nu$ be as in Theorem \ref{th laminarite sans ens. att.}.
The basin $\Br_\nu=\left\lbrace p \, | \, \dfrac{1}{n} \sum_{i=0}^n \delta_{f^i(p)} \rightharpoonup \nu  \right\rbrace$ of $\nu$ is of positive measure for $\sigma_T$. 
\end{coro}

Let us emphasis that in Theorem \ref{th laminarite sans ens. att.}, and its corollary, $\nu$ is not necessary supported on an attracting set.
It is not clear, without further assumption, that $T^s=T$ on an open set, or, that the geometric intersection $T^s\intgeom T^u$ equals $\nu=T\wedge T^u$, see section \ref{section rmk}.
Adding the conditions \ref{cond U tub}, and \ref{cond petit dt} or \ref{cond small jacobian}, to have \ref{ccl unicite}, we may use a push-pull argument and establish the following result.

\begin{theoreme}\label{th laminaire dans bassin}
Let $f$ be an endomorphism of $\proj^2$ of degree $d$. If $f$ admits a trapping region $U$, such that the conditions \ref{cond U tub}, and \ref{cond petit dt} or \ref{cond small jacobian}, are satisfied, then the Green current $T$ of $f$ is laminar and subordinate to the stable manifolds $\bigcup_{x\in \supp(\nu)} W^s(x)$ in the basin of attraction $\Br_\A=\underset{n\geq 0}{\bigcup}f^{-n}(U)$.
\end{theoreme}

\begin{coro}
Let $f$ be as in Theorem \ref{th laminaire dans bassin}  and $\nu$ be  the measure in \ref{ccl forte on nu}.
For  $\sigma_T$ almost every point $p\in\Br_\A$ $$ \dfrac{1}{n} \sum_{i=0}^n \delta_{f^i(p)} \rightharpoonup \nu .$$
\end{coro}

The difficulty in Theorem \ref{th laminarite sans ens. att.} is to prove that $T^s$ has  positive mass.
The particular case where the line at infinity $L_\infty$ is an attracting set and $\nu=T\wedge [L_\infty]$ was handled by E. Bedford and M. Jonsson \cite{bed:jon}. They use $L_\infty$ as a global transversal to bound this mass from below. 
Here, we use instead ideas of \cite{BLS}, to get the holonomy invariance along stable manifolds and the disintegration of $\nu$ on local unstable manifolds.

In \cite{deT:laminarite,DDG3}, the laminarity of the Green current is obtained
by controlling the genus of the curves $f^{-n} L$, where $L$ is a line such that $\frac{1}{d^n}f^{n*}[L]\rightarrow T$. It seems that these arguments do not apply in this setting. See Section \ref{section rmk} for more details.

\bigskip

This article is organized as follows. We start by recalling some facts about laminar currents and Pesin theory in our context. We then study the geometric structure of $\nu=T\wedge T^u$, in particular its disintegration on local unstable manifolds.

Section \ref{section construction} is devoted to the construction of the current $T^s$ from Theorem \ref{th laminarite sans ens. att.}. We then prove Theorem \ref{th laminaire dans bassin}.
In section \ref{section equidistribution}, we establish the equidistribution of periodic points in $U$ on $\nu$.
The uniqueness of the measure of maximal entropy is proved in section \ref{section uniqueness}.
Theorem \ref{th intro} is a consequence of Proposition \ref{prop conditionnelle induite par green}, Corollary \ref{coro de la laminarite}, Theorem \ref{th equi pt per} and Theorem \ref{th unicité mesure entropie max}.
We end this paper with some remarks and open questions.

\bigskip

Acknowledge : The author is grateful to Mattias Jonsson for interesting discussions on the example of section \ref{section hyp th lam bassin}. This research have been partially supported by the FMJH (Governement Program:  ANR-10-CAMP-0151-02), the ANR project LAMBDA (Governement Program: ANR-13-BS01-0002) and the NSF grant DMS-1266207.

\section{Preliminary}
\subsection{Attracting set}
Let $f$ be an endomorphisms of $\proj^2$.

\begin{prop}\label{prop A non trivial}
Let $\A$ be an attracting set for $f$ then either $\A$ is trivial, i.e. $\A=\proj^2$ or $\A$ is a finite union of attracting periodic orbits, or all trapping region of $\A$ contains a curve.
\end{prop}

\begin{proof}
Let $U$ be a trapping region of $\A$. If $U=\proj^2$ or $U$ is Kobayashi hyperbolic then $\A$ is trivial. Assume that $U$ is not Kobayashi hyperbolic, then there exists a positive closed current $\tau$ of bidegree $(1,1)$ with support on $U$. Since $U$ is an open set, by \cite[Theorem 0.1]{guedj}, $U$ contains  curves that approximate $\tau$.
\end{proof}

We also recall the definition of small topological degree attracting sets.

\begin{defi}\label{def small dt}
An attracting set $\A$ is said to be of \textbf{small topological degree} if there exists a trapping region $U$ and $n\in \N$ such that
$$\forall p\in f^n(U), \, \card(f^{-n}(p)\cap U)<d^n.$$
\end{defi}


\subsection{Assumptions}\label{section assumptions}

As mentioned in the introduction, we are going to prove each item of Theorem \ref{th intro} and Theorem \ref{th laminaire dans bassin} under weaker assumptions.
Here is the different hypotheses we will use:
\begin{enumerate}[label={$(H_\arabic*)$}]
\setcounter{enumi}{-1}
\item \label{cond att set} $U$ is a trapping region such that $ U\neq \proj^2 $ and  $U$ is not Kobayashi hyperbolic.
\end{enumerate}
We will always denote by $\A=\underset{n\in N}{\bigcap}f^n(U)$ the attracting set associated to $U$.
\begin{enumerate}[label={$(ER)$}]
\item \label{cond U rectract} $
U \text{ is a euclidean retract}\footnote{See \cite[Proposition/Definition IV 8.5]{dol}} \text{ of } \ell,
$
\end{enumerate}

\begin{enumerate}[label={$(CV)$}]
\item \label{cond unicite} 
there exists a unique invariant current $ T^u\in \Cc_{(1,1)}(U),
$ where $\Cc_{(1,1)}(U)$ is the set of positive closed currents of bidegree $(1,1)$ with supports in $U$.
\end{enumerate}
Or the weaker version:

\begin{enumerate}[label={$(CV^*)$}]
\item \label{cond CV}
There exists a current $ T^u\in C_{(1,1)}(U) \text{ such that for all } \phi$ (1,1)-form  with continuous coefficients and support in  $\Br_\A$  we have 
$\dfrac{1}{d^n}f^n_* \phi \rightarrow \langle T,\phi\rangle T^u$.
\end{enumerate}

\begin{enumerate}[label={$(H_1)$}]
\item \label{cond forte on nu} $
\nu=T\wedge T^u \text{ is mixing, of entropy }\log(d) \text{  and hyperbolic of saddle type.}$
\end{enumerate}
Or the weaker version:
\begin{enumerate}[label={$(H_1^*)$}]
\item \label{cond on nu} $
\nu=T\wedge T^u \text{ is ergodique, of entropy }\log(d) \text{  and hyperbolic of saddle type.}$
\end{enumerate}

\begin{enumerate}[label={$(H_2)$}]
\item \label{cond all hyp} 
All measures of entropy $\log(d)$ and supports on $\A$ admit a non positive Lyapunov exponent.
\end{enumerate}

All the examples known so far of attracting sets satisfy the conditions \ref{cond U tub} and \ref{cond small jacobian} but we believe that there exists a larger class of attracting sets.
For example, the assumption \ref{cond all hyp} is true as soon as the interior of  $\A$ is the empty set or $$\limsup_{n\to+\infty}\left(\int_U(f^n)^*(\omega^2)\right)^{1/n}<d,$$
see \cite{moi:taf}.

\subsection{Laminar and woven currents, geometric intersection}

We refer to \cite{dem,din:sib:survey} for general results on currents. 

\begin{defif}
A current $S$ is \textbf{uniformly woven} in an open set $U$ if there exists a constant $C>0$ such that 
$$S_{|U}=\int [M_a] \text{ d}\lambda (a)$$
where  $\lambda$ is a measure on the set of  holomorphic chains $M_a$ of area bounded by $C$. 
Such a measure is called a \textbf{marking}.
\bigskip

A current $S$ is \textbf{uniformly laminar} if for every point $x\in \supp (T)$ there exist two open sets $\B_1\subset \B_2$  biholomorphic to the  bidisk  $\D\times \D$ such that $x\in \B_1\subset \B_2$ and in good coordinates
$$S_{|\B_1}=\int [M_a\cap \B_1] \text{ d}\lambda (a)$$
where  $\lambda$ is a measure on the disk $\lbrace 0 \rbrace\times D$ and the $M_a$ are disjoints graphs $M_a=\{(x,f_a(x))\}$ in $\B_2$ such that $f_a(0)=a$.
\end{defif}

\begin{rmk}
The marking is not unique.
In fact, in $\C^2$ we have 
$$\omega=i\ds z\wedge \ds \overline{z} + i\ds w \wedge \ds \overline{w} = \frac{1}{2} i\ds (z+w) \wedge \ds (\overline{z+w}) + \frac{1}{2} i\ds (z-w) \wedge \ds (\overline{z-w}). $$
\end{rmk}

\begin{defif}\label{def laminaire}
A \textbf{woven (resp. laminar)} current is the non-decreasing limit of currents of the form
$$S_\Qr=\sum S_{Q_i}$$
where $\Qr$ is a partition in disjoints open sets (resp. open sets biholomorphic to the  bidisk) $Q_i$ and $S_{Q_i}$ is a  uniformly woven (resp. laminar) current in $Q_i$.
\end{defif}

\begin{defif}
Let $R,S$ be two uniformly laminar currents. 
We say that $R,S$ \textbf{intersect correctly} if, for every $x\in \supp(S)\cap \supp(R)$ there exists an open set $U$ biholomorphic to the bidisk such that $R,S$ are equal to
$R_{|U}=\int [M_a] \text{ d}\lambda (a)$ and $S_{|U}=\int [N_{a'}] \text{ d}\sigma (a')$,
all intersection of the graphs $M_a\cap N_{a'}$ are open sets in $M_a$ and in $N_{a'}$ and $M_a\cap \partial N_{a'}$ (resp. $\partial M_a\cap N_{a'}$) has no mass for $[M_a]$ (resp. $[N_{a'}]$).
\end{defif}

\begin{prop}[{\cite[Lemma 6.11]{BLS}}] \label{equiv Lemma 6.11}
If $T_1,\cdots,T_n$ are uniformly laminar currents which intersect correctly then $\max (T_1,\cdots,T_n)$ is a well defined laminar current.
\end{prop}

\begin{defif}\label{def int geom}
Let $D,D'$ be two disks. We define the geometrical intersection $ D \intgeom D' $ of $D$ and $D'$ has the sum on the Dirac mass of isolated  intersection  points, counted with multiplicity, of $D\cap D'$.

Let $T_1,T_2$ be uniformly woven currents in an open set $U$ and $m_1,m_2$ be marking of $T_1,T_2$. The \textbf{geometrical intersection} of $T_1$ and $T_2$ is defined by $$T_1\intgeom T_2= \int [D_1]\intgeom [D_2] \text{d} m_1\otimes m_2 (D_1,D_2).$$
\end{defif}

\begin{rmk}
Notice that this definition depends, a priori, on the chose of the markings $m_1$ and $m_2$. The following proposition says that the two definitions of intersection coincide, when there are both defined. So, in this case, the geometrical intersection does not depends on the chose of the markings, what should be the general case.
\end{rmk}

\begin{prop}\cite[Proposition 2.6]{DDG2}
Let  $T_1$ and $T_2$ be uniformly woven current, if $T_1$ or $T_2$ has bounded potentials, or more generally if $T_1\in L^1_{loc}(T_2)$ or $T_2\in L^1_{loc}(T_1)$, then $T_1 \intgeom T_2 = T_1\wedge T_2$.
\end{prop}

The geometric intersection can be use to characterize uniformly laminar current and if two uniformly laminar currents intersect correctly.

\begin{prop}\label{prop intersection nulle}
Let $T_1,T_2$ be uniformly woven currents.
If $T_1\intgeom T_1=0$ then $T_1$ is uniformly laminar and if 
$T_1\intgeom T_2=0$ then $T_1$ and $T_2$ intersect correctly.
\end{prop}

\begin{proof}
The disks in the supports of $T_1$ and $T_2$ are holomorhic disks so if $T_1\intgeom T_1=0$ (resp. $T_1\intgeom T_2=0$) then for almost every pair $(D_1,D_2)$ of disks in the support of $T_1$ (resp. in $\supp(T_1)\times \supp (T_2)$) the intersection $D_1\cap D_2$ is an open set. We conclude using the persistence of the intersection of holomorphic disks. See the beginning of the proof of \cite[Lemma 2.5]{duj:contre}.
\end{proof}

We may also define the intersection between a positive closed current $R$ of bidegree $(1,1)$ with bounded potentials and a woven current.
In fact, if $D$ is a closed disk in an open set $\Omega$,  i.e. $\Omega\cap \partial D=\emptyset$, then $R\wedge [D]$ is well defined in $\Omega$.

\begin{defif}\label{def int geom 2}
Let $R$ be a positive closed current of bidegree $(1,1)$ with bounded potentials and $S=\int [D] \text{d} m(D)$ be a uniformly woven current in an open set $\Omega\subset \C^2$. The \textbf{geometrical intersection} of $R$ and $S$ is defined by $$R\intgeom S= \int R\wedge [D] \text{d} m(D).$$ 
\end{defif}

\begin{prop}
Let $R$ be a positive closed current of bidegree $(1,1)$ with bounded potentials and $S=\int [D] \text{d} m(D)$ be a uniformly woven current in an open set $\Omega\subset \C^2$. If $R$ is also a  uniformly woven current in $\Omega$ then the definitions \ref{def int geom} and \ref{def int geom 2} coincide.
\end{prop}

\begin{proof}
This is a direct consequence of \cite[Lemma 2.7]{DDG2}.
\end{proof}


\subsection{Pesin Theory and  Lyapunov chart}
The classical presentation of Pesin theory uses the assumption that $\log (||Df||)$ is integrable, see \cite{kat:has}. We recall here a more general version of it without this assumption.
\bigskip

Let $\m$ be an invariant probability measure.
Denote by $(\hat{\proj}^2,\hf,\hat{\m})$ the natural extension of $(\proj^2,f,\m)$ which is an invertible dynamical system with the same ergodic properties than $(\proj^2,f,\m)$.
We recall that $\hat{\proj}^2= \lbrace \hx=(x_{-n})\in (\proj^2)^\N \, | \, f(x_{-n})=x_{-n+1} \rbrace$ is the set of histories with the  induced topology of $(\proj^2)^\N$. The projections $\pi_i$  and the measure $\hat{\m}$ are defined by $$\pi_i(\hx)=x_i \text{ and }\hat{\m}(\hat{A})=\lim \m(\pi_n(\hat{A})).$$
For every point $\hx$ the tangent space is define by $T_\hx \hat{\proj}^2= T_{x_0}\proj^2\simeq \C^2$ and we set $D\hf(\hx)=Df(x_0)$, see \cite{deT:expo} or \cite{rong}.
\\

We start with a more general version of Oseledets Theorem.
In the classical statement, we assume that $\log^+ (||D\hf^{\pm 1}||)\in L^1(\hat{\m})$. 
As $f$ is an endomorphism of $\proj^2$, $\log^+ (||D\hf||)\in L^1(\hat{\m})$. It turns out that the hypothesis $\log^+ (||D\hf^{- 1}||)\in L^1(\hat{\m})$ is not needed, see \cite[Appendix A1]{Kos}.

\begin{theoreme}\label{th Oseledets}
Let $f$ be an endomorphism of $\proj^2$ and $\m$ be an invariant  measure ($f_*\m=\m$).
There exist a set $\hX$ such that $\hat{\m}(\hat{\proj}^2\setminus \hX)=0$,  measurable functions (the Lyapunov exponents) $\chi_1\geq\chi_2$ and an invariant measurable  decomposition  of the tangent space such that for all $\hx\in\hX$ 
$$T_\hx \hat{\proj}^2=E_{\chi_1}(\hx)\oplus E_{\chi_2}(\hx)$$
and for all $v\in E_{\chi_i}(\hx) $, if $v\neq 0$ then 
$$\underset{n\rightarrow +\infty}{\lim} \dfrac{1}{n}\log \, ||D\hf^n(\hx)v||=\chi_i(\hx).$$
\end{theoreme}

The only difference without the assumption $\log (||D\hf^{- 1}||)\in L^1(\m)$, is that the Lyapunov exponents may be equal to $-\infty$.

\begin{rmq}
If $\m$ is ergodic then the  Lyapunov exponents are constant.
\end{rmq}

There also exists a slightly different version of the theorem of $\gamma$-reduction of Pesin, without the assumption $\log^+||(D\hf)^{-1}||\in L^1(\hat{\m})$.

\begin{theoreme}\label{Pesin faible}
Let $f$ be an endomorphism of $\proj^2$ and $\m$ be an  invariant measure ($f_*\m=\m$). Assume that the Lyapunov exponents of $\mu$ satisfies $\chi_u>0 >\chi_s$. 
Denote by 
\begin{center}
$E^u (\hx)= E_{\chi_u}(\hx)$, and $E^s (\hx)=E_{\chi_u}(\hx)$
\end{center}
For all small enough $\gamma,\varepsilon_0>0$, there exist a subset $\hY\subset \hat{\proj}^2$ of full measure and   $\gamma$-moderate functions
  $C_\gamma: \hY \rightarrow GL_2(\C)$ 
 and $\delta: \hY \rightarrow \R $ such that for almost every $\hx\in\hY$,
\begin{enumerate} 
\item  $C_\gamma (\hx)$ maps $\C^{\dim E^u(\hx)} \oplus \C^{\dim E^s(\hx)} $ on the decomposition $E^u(\hx) \oplus E^s(\hx)$,
\item  $g_\hx(w)=Exp_{\gamma,\hf(\hx)}^{-1}\circ f_\hx \circ Exp_{\gamma,\hx}(w)$ is well defined on $B(0,\delta(\hx))$, where
$$Exp_{\gamma,\hx}=exp_\hx\circ C_\gamma (\hx),$$
\item $g_\hx(0)=0$,
\item $Dg_\hx(0)= \begin{pmatrix}
   A^u_\gamma(\hx) &0 \\
   0 & A^s_\gamma(\hx) 
\end{pmatrix}$ 
\bigskip
where  $A^u_\gamma$ and $A^s_\gamma$ are functions such that we have
\begin{center}
$e^{\chi_u(\hx)-\gamma}\leq |A^u_\gamma(\hx)| \leq e^{\chi_u(\hx)+\gamma}$,
\medskip

 and $|A^s_\gamma(\hx)|\leq e^{\alpha}$ for all $\chi_s(\hx)+\gamma<\alpha<0$,
\end{center}
\item if we denote
$g_\hx(w)= D g_\hx(0)+h (w)$ then, for all $||w||\leq \delta(\hx) $, we have
\begin{center}
$||Dh_\pm(w)||\leq \varepsilon_0 $, so $||h_\pm (w)||\leq \varepsilon_0 ||w||. $
\end{center}
\end{enumerate}
\end{theoreme}

\begin{proof}
See \cite[Theorem 2.3]{newhouse}, Theorem 6.1, and Lemma 6.2 in \cite{dupont:entropy} with the notations $\gamma=\varepsilon$ et $\delta(\hx) =\varepsilon_0 (\psi_\varepsilon(\hf(\hx)))^{-1}=r$.
\end{proof}

\begin{defif}
We call  \textbf{horizontal} graph (resp. \textbf{vertical} graph), in $ Exp_{\gamma,\hx} B(\delta(\hx))$, the image under $Exp_{\gamma,\hx}$ of a  graph above $B_1(0,\delta(\hx))$ (resp. $ B_2(0,\delta(\hx))$).
\end{defif}

\begin{prop}\label{Pesin donne H like}
Assume that for $\m$-almost every $\hx$, the Lyapunov exponents satisfy $\chi_u(\hx)> 0 > \chi_s(\hx)$.
Denote by $B_u(0,r), B_s(0,r)$ the open balls of centre $0$ and radius $r$ in $E^u$ and $E^s$, 
and denote  $B(r)=B_u(0,r)\times B_s(0,r)$.

Then, for all $0< r $,  $g_\hx: B(\rho(\hx))\rightarrow B\left(\rho(\hf(\hx))\right)$, where $\rho(\hx)=\min(r,\delta(\hx))$, is a horizontal like map of degree 1 and is injective on the restriction of every   cut-off horizontal graph, i.e. a horizontal  graph in $B(\rho(\hx))\bigcap g_\hx^{-1} \left( B\left(\rho(\hf(\hx))\right) \right)$.
\end{prop}

\begin{proof}
Let  $0< r \leq \delta(\hx)$, up to divided $\delta(\hx)$ by a  constant, we may assume that Theorem \ref{Pesin faible} is true in $B(r)$.
Thus $g_\hx(0)=0$, $Dg_\hx(0)$ is a diagonal matrix and for all $w\in B(\delta(\hx))$ we have $$||Dg_\hx(w)-Dg_\hx(0)||\leq \varepsilon_0.$$

As $\chi_u>0$ and $\chi_s<0$, up to reduce $\delta(\hx)$ and $\gamma$, we may assume that
\begin{center}
$e^{\chi_u -\gamma}-\varepsilon_0>e^{\gamma}>1$ and $e^{\chi_s +\gamma}+\varepsilon_0<e^{- \gamma}<1$.
\end{center}

Denote by $p_{u}$ the projection on $E^u$ and $h(w)=g_\hx(w)-Dg_\hx(0)$ so for every $w\in B(r)$
$$
|p_u( g_\hx (w) )|  \geq ||Dg_\hx \left( p_u(w) \right)|| - ||h(w)||
  \geq e^{\chi_u - \gamma}|p_u(w)| - \varepsilon_0 r
  \geq e^\gamma r\\
$$
and 
$$
| g_\hx (w) |  \leq ||Dg_\hx (w)|| + ||h(w)||
  \leq e^{\chi_s + \gamma}r + \varepsilon_0 r
  \leq e^{- \gamma} r.
$$
If $\rho(\hx)=\min(r,\delta(\hx))$ then
 $$e^{\gamma} \rho(\hx)\geq \rho (\hf(\hx)) \geq e^{-\gamma} \rho(\hx),$$
and by the preceding facts, $g_\hx: B(\rho(\hx))\rightarrow B(\rho(\hf(\hx)))$ is a horizontal map.

Let $\Gamma$ be a  horizontal graph, i.e. $\Gamma=\lbrace (z,\varphi(z))\rbrace$ where $\varphi: B_u(\rho(\hx))\rightarrow B_s(\rho(\hx))$ is a holomorphic function.
We set $$c=1-(\varepsilon_0+e^\gamma)e^{-\chi_u+\gamma} \text{ and } r'=\rho(\hx) (\varepsilon_0+e^\gamma)e^{-\chi_u+\gamma} $$ so, thanks to the  Cauchy inequalities,  for all $z,z'\in B^u(r')$ we have $|\varphi(z) - \varphi(z')|\leq \frac{1}{c} |z-z'|$. 
 The image by $f$ of the part of the graph above $B_u(\rho(\hx)) \setminus B_u(r')$ is outside $B(\rho(\hf(\hx))$, thus we are interested only in the part of the graph above $ B_u(r')$ . 

Let $w,w' \in \Gamma$ and $z,z'\in  B_u(r')$ such that $ w=(z,\varphi(z))$,  $w'=(z',\varphi(z')) $. Since $Dg_\hx(0)$ is diagonal, we have :
$$||h(w)-h(w')||\leq 2\varepsilon_0 \max (|z-z'|,|\varphi(z)-\varphi(z')|) \leq \frac{2}{c}\varepsilon_0 \, |z-z'|$$
and
\begin{align*}
|p(g_\hx(w))-p(g_\hx(w'))| & \geq ||Dg_\hx(0)(z-z',0)|| - ||h(w)-h(w')|| \\
 & \geq \left(e^{\chi_u-\gamma}-\frac{2}{c} \varepsilon_0 \right) \, |z-z'|\\
\end{align*}

And if $\gamma$ and $\varepsilon_0$ are small enough
$$
e^{\chi_u-\gamma}-\dfrac{2}{c} \varepsilon_0 = e^{\chi_u-\gamma}- \dfrac{2\varepsilon_0}{1-(\varepsilon_0+e^\gamma)e^{-\chi_u+\gamma}}
 = \dfrac{e^{\chi_u-\gamma} - e^\gamma - 3 \varepsilon_0}{1-(\varepsilon_0+e^\gamma)e^{-\chi_u+\gamma}}
 >0.$$
Thereby, up to reduce $\delta(\hx)$, $$g_\hx: B(\delta(\hx))\cap g^{-1}( B(\delta(\hf(\hx))))\rightarrow B(\delta(\hf(\hx)))$$ is of  degree 1 and injective on every cut-off horizontal graph.
\end{proof}

\begin{rmk}
By the usual Pesin $\gamma$-reduction theorem, see Propositions 9 and 10 of \cite{deT:expo}, if $\log^+ ||(D\hf)^{-1}||\in L^1(\hat{\m})$ then $g_\hx $ is injective in all $B(\rho(\hx))$.
\end{rmk}

\begin{defif}
For $n\geq 1$, the \textbf{ graph transform} map
$$f_n : \underset{i=0,\cdots,n+1}{\bigcap} f^{-i}\left(  Exp_{\gamma,\hf^i(\hx)}  B(\delta(\hf^i(\hx))) \right) \rightarrow  Exp_{\gamma,\hf^{n+1}(\hx)} B(\delta(\hf^{n+1}(\hx)))$$ 
is the composition of $n $ cut-off maps 
$f|_{ Exp_{\gamma,\hf^i(\hx)}  B(\delta(\hf^i(\hx)))}$
where at each step we cut-off the image of $Exp_{\gamma,\hf^i(\hx)} B(\delta(\hf^i(\hx)))$ to $ Exp_{\gamma,\hf^{i+1}(\hx)} B(\delta(\hf^{i+1}(\hx)))$.
\end{defif}

\begin{prop}
For every $n\geq 1$, the preimage under $f_n$ of a  vertical graph is a  vertical graph and the image by $f_n$ of a  horizontal graph is a  horizontal graph.

Moreover, one every  horizontal graph the inverse maps $f_{-n}$ of $f_n$ is well defined.
\end{prop}

\begin{proof}
This is a direct consequence of the preceding proposition.
\end{proof}

\begin{prop}
Assume that for $\m$-almost every $\hx$, the Lyapunov exponent satisfy $\chi_u(\hx)> 0 > \chi_s(\hx)$.
Then for every $\hx \in \hY$, there exist a  unique local stable manifold $W^s_{loc}(\hx)$ and a unique local unstable manifold $W^u_{loc}(\hx)$  which are respectively  vertical and horizontal graph of $Exp_{\gamma,\hx} B(\delta(\hx))$.
Moreover, there exists a constant $c>0$ such that
\begin{align*}
& \text{for all } \, y\in W^s_{loc}(\hx), \, d(f^n(y),f^n(x_0))\leq c e^{\chi_s n}\\
& \text{for all } \, y\in W^u_{loc}(\hx), \, d(f_{-n}(y),x_{-n})\leq c e^{-\chi_u n}
\end{align*}
\end{prop}

\begin{proof}
This follows from Proposition \ref{Pesin donne H like}.
\end{proof}

  We also recall the concept of common Lyapunov chart, from \cite{BLS}, which will have an important role (see also \cite{DDG3}).
  
\begin{defi}
The   \textbf{Lyapunov chart} $L(\hat{p})$ of $\hp$ is the image by $\Exp_{\gamma,\hp}$ of the  affine bidisk of size $r(\hat{p})$ and axes $E^u (\hat{p}), E^s (p)$.
\end{defi}
 Notice that $f : L(\hat{p}) \rightarrow L( \hf (\hat{p}))$ is a  horizontal-like map of  degree 1.
We consider the sets:
\begin{equation}\label{def Ln}
L^s_n (\hat{p}) := \lbrace y \in L(\hat{p})| \, \forall 1 \leq j \leq n, \, f^j (y) \in L(\hf^j (\hp)) \rbrace\\
 \text{ and } L^u_n (\hp) := f^n L^s_n (\hf^{-n} (\hp)),
\end{equation}
and there analogues in $\hat{\proj}^2$
\begin{equation}
\hL^s_n (\hat{p}) := \lbrace \hy \in \pi_0^{-1}(L(\hat{p}))| \, \forall 1 \leq j \leq n, \, f^j (y_0) \in L(\hf^j (\hp)) \rbrace 
\\ \text{ and } \hL^u_n (\hp) := \hf^n \hL^s_n (\hf^{-n} (\hp)). \nonumber
\end{equation}
The set $L^{s/u}_n (\hp) $  converge exponentially fast to the \textbf{ local stable/unstable submanifold}  $W^{s/u}_{loc} (\hp)$.
The  local stable/unstable submanifold $\hW^{s/u}_{loc} (\hp)$ of $\hp$ in $\hat{\proj}^2$ is the  limit of $\hL^{s/u}_n(\hp) $.
Thus, depending on the  context, we will see the  local stable/unstable submanifold either as subset of $\proj^2$ or of $\hat{\proj}^2$.
Notice that $W^s(\hp)$ depend only on $p=\pi_0(\hp)$ while $W^u(\hp)$ depend on the chose of the preimages of $p$.

\begin{lemme}\label{def Rvarepsilon}
For every $\varepsilon > 0$, there exist a compact set  $\hR_\varepsilon \subset \hR$ of $\hat{m}$-measure at least  $1 - \varepsilon$, and $r>0$ such that  for every $\hp\in \hR_\varepsilon$ the submanifold $W^s_{loc}(p)$ (resp. $W^u_{loc}(\hp)$) is a graph above a disk of radius $r(p)\geq r$ in $E^s(p)$ (resp. of radius $r(\hp)\geq r$ in $E^u(\hp)$) and the restriction to $\hR_\varepsilon$ of all the preceding maps, as the directions of $E^u,E^s$ and the stable and unstable submanifolds, are continuous.
\end{lemme}

If $\hp$ and $\hat{q}$ are close enough in $\hat{\proj}^2$ then the intersection  $\hW^s_{loc} (p) \cap\ W_{loc}^u (\hat{q})$ is reduced to a  point usually denote by $[p, \hat{q} ]$. 
A subset is said to have a product structure if it is closed for $[.,.]$.
A \textbf{ Pesin box} $\hP$ is a compact subset of $\hat{\proj}^2$ of  positive measure, with a product structure such that  the size of the Lyapunov chart of every $\hx\in \hP$ is bounded from below by a positive constant.

\begin{lemme}(Bedford-Lyubich-Smillie)
For each $\eta >0$, there exists a finite family of disjoint Pesin boxes  $\hP_i$ such that $\diam \left(\pi_0(\hP_i)\right)<\eta$ and $\bigcup \hP_i$ cover $\hR_\varepsilon$ (so $\hat{\m}(\bigcup \hP_i)\geq 1-\varepsilon$).
\end{lemme}

\begin{proof}
The proof of [BLS2, Lemma 1] stay true in our situation, i.e. in $\hat{\proj}^2$, up to replace  $W^{s/u}_{loc}(\hp)$ by $\hW^{s/u}_{loc}(\hp)$
\end{proof}

\begin{prop/defi}
If $\eta>0$ is small enough then for every Pesin box $\hP_i$, satisfying $\diam \left(\pi_0(\hP_i)\right)<\eta$, there exists a   \textbf{common  Lyapunov chart} $L_i$ such that $\pi_0(\hP_i)\subset L_i \subset \underset{\hx\in \hP_i}{\bigcap} L(\hx)$, i.e.
for every $\hp\in \hP_i$, $W^s_{loc}(p_0)$ is a vertical disk in $L_i$ and  $W^u_{loc}(\hp)$ is a horizontal disk in $L_i$.
\end{prop/defi}

\begin{proof}
If $\eta>0$  is small enough, the  stable (resp. unstable) directions of points belonging to $\hP_i$ are almost parallel.
Up to reduce $\eta$, we assume that $\frac{r}{2}<r-\eta$, thus we chose $L_i$ as the image by $\Exp_{\gamma,\hp}$ of an affine  bidisk of axes $E^u (\hat{p}), E^s (p)$ and size $\frac{r}{2}$, where $\hp \in \hP_i$.
\end{proof}

\def\svgwidth{10cm}
\begin{figure}[!h]
\begin{center}
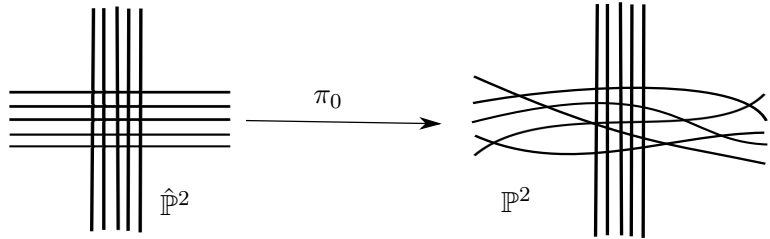
\caption{The image we have in mind of a Lyapunov chart and the stable/unstable manifolds}
\end{center}
\end{figure}

\begin{rmq}
For every $\hp\in \hP_i$, we may replace $L(\hp)$  by $L_i$. This change a little the definitions of  $L^{s/u}_n(\hp)$ (and of $\hL^{s/u}_n(\hp)$) in \eqref{def Ln} :
\begin{eqnarray}\label{new def Ln}
&\forall \hp \in \hP_i, \,
L^s_n (\hat{p}) := \lbrace y \in L_i| \, \forall 1 \leq j \leq n, \, f^j (y) \in 
L_{k_{j}}
 \rbrace \text{ where } p_j \in L_{k_{j}} \nonumber \\ 
 & \text{ and } L^u_n (\hp) := f^n L^s (\hf^{-n} (\hp)).
\end{eqnarray}
For every $\hp\in \hP_i$, denote by $W^{s/u}_{L_{i}}(\hp)$ (or   $W^{s/u}_{L}(\hp)$ or  $W^{s/u}_{loc}(\hp)$) the local  stable/unstable manifold
\begin{eqnarray}
& W^{s/u}_{L_{i}}(\hp)= \underset{n\in \N}{\bigcap} L^{s/u}_n (\hat{p}),
\end{eqnarray}
which is a vertical/horizontal disk in $L_i$.
\end{rmq}

We recall that  $f:L(\hp)\rightarrow L(\hf(\hp))$ is a  horizontal-like map of degree 1, so the local stable and unstable manifolds, defined like this, are transverse in each common Lyapunov chart.
Notice that if $\log ||d\hf^ {-1}||\notin L^1(\m) $, we cannot assume that $f:L(\hp)\rightarrow L(\hf(\hp))$ is injective.

By the Poincaré recurrence theorem, up to removing a subset of zero  measure, we assume that for every $ \hx\in \Rr_\varepsilon$,  there exist infinitely many $n>0$ such that $\hf^n(x)\in \Rr_\varepsilon$.
And, up to remove a subset of $\varepsilon$ measure to $\R_\varepsilon$, there exists $C <\infty$ such that, for every $\hP_j$ and every $\hx \in\hP_j$, the following  properties are true :
\begin{align}
\label{contraction} &\dist(f^n(x_0),f^n(y_0)) < Ce^{-n(\chi^s-\gamma)} \text{ for } n \geq 1 \text{ and } y_0\in W^s_{L_j}(x_0),\\
\label{dilatation}&\dist(x_{-n},y_{-n})	< Ce^{-n(\chi^u-\gamma)} 
\text{ for } n > 1 \text{ and } \hat{y}\in W^u_{L_j}(\hat{x}).
\end{align}
Let $\Delta$ be a disk, it follow from the proof of \cite[Proposition 5.10]{FS2}  that $T\intgeom [\Delta]=0$ iff $(f^n_{|\Delta})$ is normal. Thus, for  $\m$-almost every $\hx\in \Rr$, $T\llcorner_{W^u_{L(\hx)}(\hat{x})}:=T\intgeom[W^u_{L(\hx)}(\hat{x})]$ is a  positive measure. Up to remove a subset of $\varepsilon$ measure to  $\R_\varepsilon$, there exists $m_0> 0$ such that for every $\hx\in \hP_j\subset \R_\varepsilon$

\begin{align}
\label{mesure positive}
&T\llcorner_{W^u_{L_j}(\hat{x})} (W^u_{L_j}(\hat{x})) \geq m_0.
\end{align}

For every $F \subset \hP \subset\Rr_\varepsilon \subset \{\hat{x}\in \hat{\Rr}: r(\hat{x}) \geq r\}$, denote by $W^{s/u}_{L}(F)$ the union of local stable/unstable manifolds of $F$
\begin{eqnarray}
W^{s/u}_{L}(F)= \underset{\hat{x}\in F}{\bigcup} W^{s/u}_{L}(\hat{x}).
\end{eqnarray}

\section{Conditional measures on unstable manifolds}\label{section structure geometrique}

Let $\m$ be an invariant (under $f$) probability measure  and $\Br$ alors $\supp(\m)\subset \A$. 
$\xi$ be a measurable  partition of $\proj^2$. The measure $\m$ may be disintegrated with respect to $\xi$, i.e.  the conditional measures $\m(\cdot|\xi(x))$ of $\m$ on the fibers of  $\xi$ are well defined, see \cite{BLS} and the references in it.

Denote by $J_{\m,\xi}f$ the Jacobian of $f$ with respect to the conditional measures of $\m$ on the fibers of $\xi$, i.e. the Radon-Nikodym derivative of $f$ with respect to $\m(\cdot|\xi(x))$
$$J_{\m,\xi}f(x)= \dfrac{\ds f^*\m(\cdot|\xi(f(x)))}{\ds \m(\cdot|\xi(x))}=\dfrac{1}{\m\left( f^{-1}(\xi(f(x)))|\xi(x) \right)}$$
since $\m$ is invariant under $f$.
Denote by $p(x)=\m((f^{-1}\xi)(x)|\xi(x))$ and $h_{\m}(f,\xi)$ the entropy of $\m$ with respect to $\xi$, we have:
\begin{equation}\label{egalite entropie int log(p)}
h_{\m}(f,\xi)=-\int \log(p(x))\ds \m(x)=\int \log(J_{\m,\xi}f(x))\ds \m(x).
\end{equation}
The partition $\xi$ is said to be \emph{generating} if the partition $\xi^\infty=\bigvee_{n=0}^\infty f^n \xi$ is the partition generated by singletons. In this case, we have $h_{\m}(f)=h_\m(f,\xi)$.

Assume that $\supp(\m)\cap \supp(\mu_{eq})=\emptyset$ and $\m$ is of (maximal) entropy $\log(d)$.
Then, by Margulis-Ruelle inequality,  $\m$ admit a positive  Lyapunov exponent and $\hm$-almost every point $\hx\in\hat{\proj}^2$ admit  an Pesin unstable manifold $\hat{W}^u(\hx)$.
Assume also that $\m$ admit a negative exponent.

\begin{defi}
A measurable partition $\hxi$ of $\hat{\proj}^2$ is said to be subordinate to  the unstable manifolds of $(\hf,\hm)$ if for $\hm$-a.e. $\hx\in\hat{\proj}^2$, $\hxi(x)$ has the following properties:
\begin{enumerate}
\item  $\pi_{0|\hxi(\hx)}: \hxi(\hx) \rightarrow \pi_0(\hxi(\hx))$ is bijective,
\item $\pi_0(\hxi(\hx))\subset W^u(\hx)$ and $\pi_0(\hxi(\hx))$ contains an open neighborhood of $x_0$ in $W^u(\hx)$.
\end{enumerate} 
\end{defi}

\begin{prop}\label{prop existence partition sub à W^u}
Let $\m$ be a $f$-invariant  measure with support outside $\supp(\mu_{eq})$, of (maximal) entropy $\log(d)$. If for $\m$-a.e. $x$ the Lyapunov exponent satisfy $\chi_u(x)>0\geq \chi_s(x)$  then there exists a generating partition $\hxi^u$ of $\hat{\proj}^2$ subordinate to Pesin unstable manifolds. 
Moreover, for all $\hx$ we have $$\left(\hf^{-1}\hxi^u\right)(\hx)=\hf^{-1}\left(\hxi^u(\hf(\hx))\right)\subset \hxi^u(\hx).$$
\end{prop}

This proposition follows from Proposition 3.2 of \cite{qian:zhu}, see also \cite{led:str}. In fact, in our case all the properties of the Pesin theory are  satisfied except that only the restriction of $f$ to horizontal disks, and not to the entire  Lyapunov box, is injective but this is not needed in the proof.

Up to refine $\hxi^u$, we may assume that for $\hm$-almost every $\hx$,  $\pi_0(\hxi(\hx))$ is included in a Lyapunov box.  In particular, the map $\pi_0:\hxi(\hx)\rightarrow W^u(\hx)$ is injective.
We have the following proposition analogous to \cite[Proposition 3.2]{BLS}.

\begin{prop}\label{prop conditionnelle induite par green}
Let $\m$ be a $f$-invariant  measure with support outside $\supp(\mu_{eq})$, of (maximal) entropy $\log(d)$.
If for $\m$-a.e. $x$ the Lyapunov exponent satisfy $\chi_u(x)>0\geq \chi_s(x)$ then the conditional measures of $\hm$ on $\hxi^u$ are induced by the Green current $T$, i.e. for  $\hm$-a.e. $\hx$  we have 
$$\hm(\cdot|\hxi^u(\hx))=\dfrac{(\pi_0^{-1})_*\left( T\intgeom [D^u(\hx)] \right)}{M\left( (\pi_0^{-1})_*\left( T\intgeom [D^u(\hx)] \right) \right)}$$
where $D^u(\hx)=\pi_0\left(\hxi^u(\hx)\right)$.
\end{prop}

\begin{proof}
Let $\hx$ be in the set of  Pesin regular points $\Rr$.
Denote by  $$\rho(\hx)=M\left((\pi_0^{-1})_*\left(T\intgeom [D^u(\hx)]\right)\right)=M(T\intgeom [D^u(\hx)])$$ where $D^u(\hx)=\pi_0\left(\hxi^u(\hx)\right)$.
 Since $D^u(\hx)$ is not a Fatou disk, $M(T\intgeom [D^u(\hx)])=\rho(\hx)>0$, see \cite[Proposition 5.10]{FS2}.
Thereby we may normalise the family $\left((\pi_0^{-1})_*\left(T\intgeom [D^u(\hx)]\right)\right)$ by $\rho(\hx)$, to obtain a family of probability measures
$$\mu_\hx:=\dfrac{(\pi_0^{-1})_*\left(T\intgeom [D^u(\hx)]\right)}{\rho(x)}.$$
Denote by $q(\hx)=\mu_\hx\left(\hf^{-1}(\hxi^u(f(x)))\right)$. By the Pesin theory, $f:f^{-1}\left(W^u_{loc}(\hf(\hx))\right)\cap  W^u_{loc}(\hx)\rightarrow W^u_{loc}(\hf(\hx))$ is injective, so we have 
\begin{align*}
q(\hx)&=\dfrac{\left( T\intgeom [D^u(\hx)] \right) \left( \pi_0\left(\hf^{-1}(\hxi^u(\hf(\hx)))\right) \right)}{ \rho(\hx)}&\\
& =\dfrac{\left( T\intgeom [D^u(\hx)] \right) \left( f^{-1}\left(\pi_0(\hxi^u(\hf(\hx)))\right) \right)}{ \rho(\hx)} &\\
& = \dfrac{1}{d}\cdot\dfrac{\left( T\intgeom [D^u(\hf(\hx))] \right) \left( \pi_0\left(\hxi^u(\hf(\hx))\right) \right)}{\rho(\hx)} &\text{ since } \dfrac{1}{d}f^*T=T\\
& =\dfrac{1}{d}\cdot\dfrac{\rho(\hf(\hx))}{ \rho(\hx)} &
\end{align*}
and then
$$\log (q(\hx))=\log\left(\rho(\hf(\hx))\right) - \log(\rho(\hx)) - \log (d).$$
Since $\hf^{-1}\left(\hxi^u(\hf(\hx))\right)\subset \hxi^u(\hx)$ and  $f^*T=d\, T$, we have
$$M\left(\rho(\hf(\hx))\right)=M(T\intgeom [D^u(\hf(\hx))])\leq M(T\intgeom [f(D^u(\hx)])=d \, M(T\intgeom [D^u(\hx)]),$$ and thus $\rho(\hf(\hx))\leq d \, \rho(\hx)$ et $q(\hx)\leq 1$.
The measure $\hm$ is $f$-invariant and $\log \circ \rho$ is bounded from above, so
\begin{equation}\label{int q(x)}
-\int \log (q(\hx)) \ds \hm(\hx)=\log(d),
\end{equation}
see \cite[Lemma 2.7.]{BLS}.

On the other hand, denote by $p(\hx)=\hm\left(\hf^{-1}\left(\hxi^u(\hf(\hx))\right)|\hxi^u(\hx)\right)$, since $\hxi^u$ is a generator (Proposition \ref{prop existence partition sub à W^u}), by \eqref{egalite entropie int log(p)} we have
\begin{equation}\label{int p(x)}
-\int \log \left(p(\hx)\right) \ds \hm(\hx)=h_{\hm}(f,\hxi^u)=h_{\hm}(f)=\log (d)
\end{equation}
We deduce from \eqref{int q(x)} and \eqref{int p(x)} that
$$\int \log \left( \dfrac{q(\hx)}{p(\hx)} \right)=0.$$
For $\hm$-almost every $\hx$ we have $ \hf^{-1}(\hxi^u(f(\hx)))\subset \hxi^u(\hx)$ and if $ \hf^{-1}(\hxi^u(f(\hy)))\subset \hxi^u(\hx)$ then $ \hxi^u(\hy)= \hxi^u(\hx)$.
By definition, $p$ and $q$ are constant on each disk of the form $\hf^{-1}(\hxi^u(f(\hy)))\cap  \hxi^u(\hx)$.
If  $\sum_{ \hf^{-1}(\hxi^u(f(\hy))) \subset \hxi^u(\hx)}$ denote the sum indexed on the set of disks of the  form  $ \hf^{-1}(\hxi^u(f(\hy)))$  included in  $\hxi^u(\hx)$ then we have 
\begin{eqnarray*}
\int \dfrac{q(\hx)}{p(\hx)} \ds \hm(\hx) & =& \int \left( \int_{\hxi^u(\hx)} \dfrac{q(\hx)}{p(\hx)} \ds \hm(\hx|\hxi^u(\hx)) \right)\ds \hm(\hx)\\
& =& \int \left(  \sum_{ \hf^{-1}(\hxi^u(f(\hy))) \subset \hxi^u(\hx)} \int_{\hf^{-1}(\hxi^u(f(\hy))} \dfrac{q(\hy)}{p(\hy)} \ds \hm(\hx|\hxi^u(\hx)) \right)\ds \hm(\hx)\\
& =& \int \left(  \sum_{ \hf^{-1}(\hxi^u(f(\hy))) \subset \hxi^u(\hx)}\dfrac{q(\hy)}{p(\hy)}  \hm(\hf^{-1}(\hxi^u(f(\hy))|\hxi^u(\hx)) \right)\ds \hm(\hx)\\
&= &\int \left(  \sum_{ \hf^{-1}(\hxi^u(f(\hy))) \subset \hxi^u(\hx)}\dfrac{q(\hy)}{p(\hy)}  \hm(\hf^{-1}(\hxi^u(f(\hy))|\hxi^u(\hy)) \right)\ds \hm(\hx)\\
& = &\int \left(  \sum_{ \hf^{-1}(\hxi^u(f(\hy))) \subset \hxi^u(\hx)}\dfrac{q(\hy)}{p(\hy)}  \times p(\hy) \right)\ds \hm(\hx)\\
& =&  \int \left(  \sum_{ \hf^{-1}(\hxi^u(f(\hy))) \subset \hxi^u(\hx)} q(\hy)\right)\ds \hm(\hx)\\
& = &\int 1\ds \hm (\hx)\\
&=&1
\end{eqnarray*}

Since $\log$ is concave, we obtain that for $\hm$-almost every $\hx$ we have $p(\hx)=q(\hx)$, i.e. for 
$\hm$-almost every $\hx$
$$\hm\left(\hf^{-1}\left(\hxi^u(\hf(\hx))\right)|\hxi^u(\hx)\right)= \mu_\hx\left(\hf^{-1}(\hxi^u(f(x)))\right).$$
Applying this to $\hf^n$, and since $\bigvee_{n\in\N} \hf^{-n}\hxi^u$  is the partition in singletons, we get:
$$\hm\left(\hf^{-n}\left(\hxi^u(\hf^n(\hx))\right)|\hxi^u(\hx)\right))= \mu_\hx\left(\hf^{-n}(\hxi^u(\hf^n(\hx)))\right)$$
and 
$$\hm(\cdot|\hxi^u(\hx))= \mu_\hx=\dfrac{(\pi_0^{-1})_*\left(T\intgeom [D^u(\hx)]\right)}{M\left( (\pi_0^{-1})_*\left(T\intgeom [D^u(\hx)]\right) \right)},$$
where $D^u(\hx)=\pi_0\left(\hxi^u(\hx)\right)$.
\end{proof}


 \section{Proof of Theorem \ref{th laminarite sans ens. att.}}\label{section construction}

We start with the construction of a laminar current $T^s_\hP\leq T$ subordinate to Pesin unstable manifolds while, in Theorem \ref{th laminarite sans ens. att.}, $W^s(x)\subset \supp(T^s)$ for $\nu$-a.e. $x$.
 We fix a Pesin box $\hP_i$ and a common Lyapunov chart $L_i$ and denote it by $\hP $ and $L$. 
 
\begin{theoreme}\label{constr. TshP}
If there exists a $f$-invariant measure  $\nu$ of entropy $\log(d)$ such that $\supp(\nu)\cap\supp(\mu_{eq})= \emptyset$ and its Lyapunov exponents satisfy $\chi_u>0>\chi_s$ then in each Pesin box $\hP$, there exists a positive current $T^s_\hP$ of bidegree $(1,1)$, uniformly laminar, of positive mass such that  $T^s_\hP\leq T$ and such that $T^s_\hP$ is closed in $L$. Moreover, $T^s_\hP$ is subordinate to $W^s_L(\hP)$ and  $W^s_L(\hp)= \supp(T^s_\hP)$.
\end{theoreme}

\begin{rmk}
In this theorem we do not assume that $\nu$ is ergodic.
\end{rmk}

By Proposition \ref{prop conditionnelle induite par green}, the disintegration of $\hnu$ on $\hxi^u$ is induced by the Green current. We are going to prove that they are invariant by holonomy in the common Lyapunov chart $L$ in order to estimate the number of ``tubes'' of the form $L^s_n$, and, thereby, assure that $T^s_\hP$ has positive mass.

\subsection{Holonomy invariance in the common Lyapunov chart $L_i$}

Let $\Wr^s_L$ be the family of stable manifolds (which are pieces of complex manifolds) of points in $\hP $, i.e. $\Wr^s_L=\underset{\hp\in \hP}{\bigcup}W^s_{L}(\hp) $, let $\Dr^{horiz}$ be the set of horizontal disks in $L$ transverse to $\Wr^s_L$, i.e. a horizontal disk $D$ is in $\Dr^{horiz}$ iff $D$ intersects each $W^s\in \Wr^s_L$ in a unique point, and this intersection is transverse.
 By the Pesin theory, we have $\underset{\hp\in \hP}{\bigcup}W^u_{L}(\hp) \subset \Dr^{horiz}$.
Let $D$ and $D'$ be two disks in $\Dr^{horiz}$, denote by
$$X = \underset{W^s\in \Wr^s_L}{\bigcup} D\cap W^s \text{ and } X' = \underset{W^s\in \Wr^s_L}{\bigcup} D'\cap W^s$$
 We define the \textit{holonomy map} $$\hol:= \hol(D, D', \Wr^s_L):X \rightarrow X'$$
by $\hol(x)= W^s(x)\cap D'$, where $W^s(x)\in \Wr^s_L$ is the unique stable manifold containing $x$.

The holonomy map $\hol(D, D', \Fr^s)$ is well defined since $D$ and $D'$ are disks in $\Dr^{horiz}$ and $L$ is a common Lyapunov chart.

\begin{prop}\label{inv holonomie}
There is holonomy invariance in $L$, i.e.
for all disks $D,D'$ transverse to $W^s_{L}(\hP)$ we have  $\hol(T\intgeom [D]|_{X})=T\intgeom [D]|_{hol(X)}$.
\end{prop}

\begin{proof}
We start with a local proof in a neighborhood of a point $\pi_0(\hp)$, where $\hp \in \hP$, and then we use a covering argument to obtain the full result.

Recall that $\hp\in\hP\subset \Rr_\varepsilon$.
Let $a,a'$ be the  intersection   points $\{a\}= W^s_{L}(\hp)\cap D$ and $a'=\hol(a)\in  W^s_{L}(\hp) \cap D'$, and, let $n$ be such that $\hf^n(\hp)\in \Rr_\varepsilon$.
Then there  exist $j$ such that $\hf^n(\hp) \in P_j$.
Denote by $D_n,D'_n\subset L_j$ the cut-off images of $D,D'$ by $f_n:L \rightarrow L_j$. 
The disks  $D_n,D'_n$ are transverse to $\Wr^s_{L,n}=f_n(\Wr^s_L)$ and the restriction of $f_n$ to $f_n^{-1}(D_n)$ (resp. $f_n^{-1}(D'_n)$)  admits an holomorphic inverse $f_{-n}$ with value in $D$ (resp. $D'$) such that $f_{-n}(D_n)=D\cap L^s_n(\hp)$ (resp. $f_{-n}(D_n')=D'\cap L^s_n(\hp)$), see Proposition \ref{Pesin donne H like}. We have $f_{-n}(D_n)= D\cap L^s_n(\hp)$ and   $f_{-n}(D')= D\cap L^s_n(\hp)$. 
Denote by $$\hol_n: X_n \rightarrow X'_n$$ the holonomy map between $X_n=D_n\cap \Wr^s_{L,n}=f_n(X)$ and $X'_n=D'_n\cap \Wr^s_{L,n}=f_n(X')$.\\

Through the end of this section, denote by $r=\delta(\hp)$ the ``size'' of the local stable and unstable manifolds, and $\lambda= \chi_s-\gamma$, see Theorem \ref{Pesin faible}.
We have the analogous of \cite[Lemme 4.1]{BLS}:
\begin{lemme}\label{equiv 4.1}
With the preceding notations, we have $\hol_n\circ f_n=f_n\circ \hol$ and for $r_0<r/4$ 
$$\hol_n(X_n\cap B(f_n(a),r_0-Ce^{-n\lambda})) \subset \hol_n(X_n)\cap B(f_n(a'),r_0) \subset \hol_n(X_n\cap B(f_n(a),r_0+Ce^{-n\lambda})).$$
\end{lemme}

Since $T$ has a continuous potential, we have:
\begin{lemme}\label{equiv 4.3}
If $r/8<r_0<r/4$ then there exists a sequence $(n_k)$ of integers such that
$$\lim \frac{T\intgeom D_{n_k} (B(f_{n_k}(a),r_0\pm Ce^{-{n_k}\lambda}))}{T\intgeom D'_{n_k} (B(f_{n_k}(a'),r_0\mp Ce^{-{n_k}\lambda}))}=1. $$
\end{lemme}

\begin{proof}
The proof follows from \cite[Lemma 4.2]{BLS} and \eqref{mesure positive}.
\end{proof}

To finish the proof of Proposition \ref{inv holonomie}, we construct a decreasing family of neighbourhood $C^{\pm}_n(a)$ and $C_n'(a')$ of $a\in D$ and $a'\in D'$, and use the preceding Lemma as Bedford-Lyubich-Smillie did in the proof of \cite[Lemma 4.4]{BLS}.

Since $\hp,\hf^n(\hp)\in\Rr_\varepsilon$, the disks $D,D'$ (resp. $D_n,D_n'$) are graphs above  disks of radius $r$ in $E^u(\hp)$ (resp. $E^u(\hf^n(\hp))$). 

Let $h$ (resp. $h'$) be  holomorphic functions such that $D$ (resp. $D'$) is the image by $h$ (resp. $h'$) of a flat disk.
Thanks to the Koebe  Distorsion theorem, replacing ``$f^{-n}: W^u_r(f^n(a))\rightarrow W^u_r(a)$'' by  $f_{-n}:D_n\rightarrow D$ (resp. $f_{-n}:D_n'\rightarrow D'$) in the proof of  \cite[Lemma 4.4]{BLS}, we get that $C^{\pm}_n(a)=f_{-n}(D_n\cap B(f_n(a),r_0\pm Ce^{-n\lambda}))$ (resp. $C_n'(a')=f_{-n}(D'_n\cap B(f_n(a'),r_0))$) is the image by $h$ (resp. $h'$) of a convex set.

And thanks to Lemma \ref{equiv 4.1}, we have $\hol (C^{-}_n(a)\cap X) \subset C_n'(a)\cap X' \subset \hol ( C^{+}_n(a)\cap X)  $.  
Since $f^n(D_n)=D$, we have  $T\intgeom [D]=T\intgeom f^n_*[D_n]=f^n_*(f^{n*}T\intgeom [D_n])=d^n f^n_*(T\intgeom [D_n])$ and we deduce from Lemma \ref{equiv 4.3} that 
\begin{equation}\label{equiv (4.11)}
\lim \frac{T\intgeom D (C^{\pm}_n(a))}{T\intgeom D' (C_n'(a))}=1. 
\end{equation}

Let $E\subset X$ be a compact set and $E'=\hol (E)$, so for every $\delta>0$ there exists an open set $O$ of $X$ such that $$T\intgeom D(O)\leq T\intgeom D(E)+\delta.$$
The set $\Cscr^{\pm}=\{C^{\pm}_n(a) \, | \, a\in X, n \text{ such that } \hf^n(\hp)\in \Rr_\varepsilon\}$ (resp. $\Cscr'=\{ C_n'(a')) \, | \, a'\in X', n \text{ such that } \hf^n(\hp)\in \Rr_\varepsilon\}$) is a neighbourhood basis of the points of $X$ in $D$ (resp. of $X'$ in $D'$) and the image by $h$ (resp. $h'$) of convex sets. By the  Morse cover theorem \cite{morse}, we deduce that there exists a family $\{ C'_j:j=1,2,\ldots \} \subset \Cscr'$  of non-overlapping open subsets of $D'$ such that
\begin{equation}\label{equiv (4.12)}
 T\intgeom D'\left(E'-\bigcup_{j} C'_j \right)=0.
\end{equation}
For every $j\in \N$, there exist $a'_j=\hol(a_j)$ and $n_j$ such that $C'_j=C'_{n_j}(a'_j)$ and the corresponding $C_j^-$ (i.e. $C_j^-=C_{n_j}(a_j)$) form also a  family of non-overlapping open subsets that belong to $\Cscr^-$ and satisfy $C_j\subset \hol^{-1}(C'_j)$. 
The diameter of the $C'_j$ can be chosen as small as wanted and $\hol^{-1}$ is continuous so we assume that $C_j^-\subset O $, thus $$T\intgeom D(O) \geq \sum_j T\intgeom D(C^-_j)$$
and, by \eqref{equiv (4.11)} and \eqref{equiv (4.12)}, we have 
$$T\intgeom D(O)\geq \dfrac{1}{1+\delta} \sum_j T\intgeom D'(C'_j) = \dfrac{1}{1+\delta} T\intgeom D'(E') .$$
For every $\delta>0$, we have $ T\intgeom D'(E') +\delta \geq \dfrac{1}{1+\delta} T\intgeom D'(E') $, thus
$$T\intgeom D(E)  \geq  T\intgeom D'(E').$$
Similarly, by covering $E$ by a family of $\Cscr^+$, we show that $T\intgeom D'(E')  \geq  T\intgeom D(E)$ ; and conclude that $T\intgeom D'(E')  =  T\intgeom D(E).$ 
This finish the proof Proposition \ref{inv holonomie}.
\end{proof}


\subsection{Number of non-overlapping tubes of the form $L^s_n(\hp)$ in $L$}

Up to reduce the size of the common  Lyapunov chart $L$, we assume that $\nu(\partial L)=0$.

\begin{lemme}
The set $\N_\hP=\lbrace n\in \N \,|\, \hnu (\hf^{-n}(\hP)\cap \hP) \geq \hnu(\hP)^2(1-\varepsilon) \rbrace$ is infinite.
\end{lemme}

\begin{proof}
Assume that $\N_\hP$ is finite, then
 $$\frac{1}{m}\underset{n=0}{\overset{m-1}{\sum}} \hnu (\hf^{-n}(\hP)\cap \hP) \leq \frac{1}{m}\underset{\underset{n\in \N_\hP}{n=0}}{\overset{m-1}{\sum}} \hnu (\hf^{-n}(\hP)\cap \hP) + \hnu(\hP)^2(1-\varepsilon), $$
so $$\underset{m\rightarrow \infty}{\limsup}\frac{1}{m}\underset{n=0}{\overset{m-1}{\sum}} \hnu (\hf^{-n}(\hP)\cap \hP) \leq \hnu(\hP)^2(1-\varepsilon).$$
On the other hand, as in the proof of  \cite[Lemma 5.1.]{deT:selles}, we may  decompose $\hnu$ in ergodic measures to get that
$$\hnu(\hP)^2 \leq \underset{m\rightarrow \infty}{\limsup}\frac{1}{m}\underset{n=0}{\overset{m-1}{\sum}} \hnu (\hf^{-n}(\hP)\cap \hP)$$
In fact, this is true for   ergodic measures and $x\mapsto x^2$ is convex.
We rich a contradiction.
\end{proof}

Notice that, for every $\hp\in \hf^{-n}(\hP)\cap \hP$, Proposition \ref{inv holonomie} is true by replacing $L(\hp)$ by $L^s_n(\hp)$ 

In fact, for every $\hq\in \pi_0^{-1}(L^s_n(\hp))\cap \hP$ the local stable manifold $W^s_L(\hq)$ is  included in  $L^s_n(\hp)$, thereby we have holonomy invariance.

Recall the following notations: for every $\hp\in\hR$,  $f_{-n}$ denote the cut-off map on a horizontal disk of $L(\hp)$ to a horizontal disk of $L^u_n(\hf^{-n}(\hp))$ and for every subset $F$ of a Pesin box $W^{s/u}(F)$ denote $W^{s/u}(F)=\underset{x\in F}{\bigcup}W^{s/u}(x)$.\\

\begin{lemme}\label{lemme Lsn= compo connex}
The  connected component of $L\cap f^{-n}(L)$ which contained  $\pi_0(\hp)$ is $L^s_n(\hp)$.
In particular, if $\hp,\hq\in \hP\cap f^{-n}(\hP) $ and $L^s_n(\hp)\cap L^s_n(\hq)$ is non-empty then $L^s_n(\hp) = L^s_n(\hq)$.
 \end{lemme}
 
 \begin{proof} 
Assume that  there exists a point $q$ which is not in  $L^s_n(\hp)$ but is in the connected component of $L\cap f^{-n}(L)$ which contained $\pi_0(\hp)$.
The connected component of $L\cap f^{-1}(L)$ is connected by arcs and contains $L^s_n(\hp)$.
Let $\rho$ be a path from $q$ to $\pi_0(\hp)$ in $L\cap f^{-n}(L)$.
Let $i$ be the smallest positive integer such that $f^i(q)\notin L(\hf^i(\hp))$. 
Since $f$ is a horizontal-like map of degree 1 between the Lyapunov charts, the path $f^i\circ \rho$ cross the  vertical side of $L(\hf^i(\hp))$ and so the path $f^n\circ \rho$ cross the  vertical side of  $L(\hf^n(\hp))=L$. 
This contradicts the fact that $\rho$ is a path in $L\cap f^{-n}(L)$.
 \end{proof}

\begin{lemme}\label{majoration nu_{|L^s_n}}
For every Pesin box $\hP$, $n\in \N$ and $\hp\in \hf^{-n}(\hP)\cap \hP$  we have
$$\hnu\left(\pi_0^{-1}(L^s_n(\hp))\bigcap \hP\bigcap \hf^{-n}(\hP)\right) \leq d^{-n} \hnu(\hP). $$
\end{lemme}

\begin{proof}
Denote $D, D_{n}$ the horizontal disks $D= W^u_L(\hp)$ and $D_{n}=W^u_L(\hf^n(\hp))$.
Since $\hp\in \hf^{-n}(\hP)\cap \hP$, we have $f^n(D\cap L^s_n(\hp))=D_n$ so 
$$f^n_*\left( T\intgeom [D\cap L^s_n(\hp)] \right)= \dfrac{1}{d^n} T\intgeom f^n_*[D\cap L^s_n(\hp)]  =  \dfrac{1}{d^n}T\intgeom [D_{n}] .$$
By Lemma \ref{lemme Lsn= compo connex}, for every $\hq\in \hP$ either $W^s_L(\hq)\cap L^s_n(\hp)= W^s_L(\hq)$ or $W^s_L(\hq)\cap L^s_n(\hp)= \emptyset$.
Thus, by Proposition \ref{inv holonomie}, we have
$T\intgeom [D_{n}]|_{\Wr^s_L }= \hol_{D,D_n \, *} \left( T\intgeom [D]|_{\Wr^s_L }\right),$
and 
$$f^n_*\left( T\intgeom [D\cap L^s_n(\hp)] \right)|_{\Wr^s_L }= d^{-n} \hol_{D,D_n \, *} \left( T\intgeom [D]|_{\Wr^s_L }\right). $$ 
Since $\pi_0\left(\hP\bigcap \hf^{-n}(\hP)\right)= 
\pi_0\left(\hf^{-n}\left(\hP\cap \hf^n(\hP)\right)\right) \subset f^{-n}\left(\pi_0(\hP\bigcap \hf^{n}(\hP))\right) $, we have
\begin{eqnarray*}
T\intgeom [D]|_{\Wr^s_L }\left(\pi_0\left(\hP\bigcap \hf^{-n}(\hP)\right) \bigcap L^s_n(\hp) \right)
&\leq &T\intgeom [D]|_{\Wr^s_L }\left( f^{-n}\left(\pi_0(\hP\bigcap \hf^{n}(\hP))\right) \bigcap L^s_n(\hp)\right)\\
&\leq & f^n_*\left( T\intgeom [D\cap L^s_n(\hp)] \right)|_{\Wr^s_L } \left(\pi_0(\hP\bigcap \hf^{n}(\hP))\right)\\
&\leq &  d^{-n} \hol_{D,D_n \, *} \left( T\intgeom [D]|_{\Wr^s_L }\right)\left(\pi_0(\hP\bigcap \hf^{n}(\hP))\right) \\
&\leq & d^{-n} T\intgeom [D](\pi_0(\hP)).
\end{eqnarray*}
Notice that $\hol_{D,D_n}^{-1}\left(\pi_0(\hP\bigcap \hf^{n}(\hP))\right)$ is included in $ D\cap \pi_0(\hP)$ but is not necessarily included in $ D\cap \pi_0(\hP\cap f^{-n}(\hP))$.

The conditionals of $\hnu$ with respect to $\hxi^u$ are induced by $T$, see Proposition \ref{prop conditionnelle induite par green}, thus for every $n\in \N_\hP$ and every $\hp\in \hf^{-n}(\hP)\cap \hP$, we have
$$\hnu\left(\pi_0^{-1}\left(L^s_n(\hp)\right)\bigcap \hP\bigcap \hf^{-n}(\hP)\,\left|\, \hxi^u(\hp) \right. \right) 
\leq d^{-n} \hnu(\hP\,|\, \hxi^u(\hp) ), $$ 
and, thereby, $\hnu\left(\pi_0^{-1}(L^s_n(\hp))\bigcap \hP\bigcap \hf^{-n}(\hP) \right) \leq d^{-n} \hnu(\hP). $
\end{proof}

\begin{prop}\label{nb de tubes}
For every Pesin box $\hP$ and every $n\in \N_\hP$, there is at least $ d^n\hnu(\hP)(1-\varepsilon)$ non-overlapping tubes  of the form $L^s_n(\hp)$ with $\hp\in \hf^n(\hP)\cap \hP$.
\end{prop}

\begin{proof}
For every $\hp\in \hP\cap \hf^{-n}(\hP)$ we have $\hnu\left(\pi_0^{-1}(L^s_n(\hp))\cap \hP\cap \hf^{-n}(\hP)\right) \leq d^{-n} \hnu(\hP)$. Since 
$$\hP\cap \hf^{-n}(\hP)=\underset{\hp\in \hP}{\bigcup}\pi_0^{-1}\left(L^s_n(\hp)\right)\cap \hP\cap \hf^{-n}(\hP),$$
and for every $n\in \N_\hP$, we have $\hnu(\hP\cap \hf^{-n}(\hP))\geq \hnu(\hP)^2(1-\varepsilon)$, we conclude with Lemma \ref{lemme Lsn= compo connex}.
\end{proof}


\subsection{Proof of Theorem \ref{constr. TshP} and a corollary}
Recall that $\hP$ is a Pesin box, $L$ is a common  Lyapunov chart of $\hP$ and for every $ \hp \in \hP$,  $L^s_n(\hp)$ and $L^u_n(\hp)$ are defined by \eqref{new def Ln}.

\begin{proof}[Proof of Theorem \ref{constr. TshP}] 
Let $l$ be a  transverse line to  $W^u_{L}(\hP)$ (i.e. to $W^u_{L}(\hp)$ for every $\hp\in \hP$) such that the disk $\Delta=l\cap L$ is vertical.
We may chose $l$ such that $\frac{1}{d^n}f^n_*[l]\rightarrow T$.

Fix $\varepsilon>0$ and $n\in \N_\hP$. 
For every $\hp\in \hf^{-n}(\hP)\cap \hP$, we have $L(\hp)=L(\hf^n(\hp))=L$, so $\Delta$ can be seen as a  vertical disk of $L(\hf^n(\hp))$.
Denote by  $\Delta_{n,\hp}$ the cut-off preimage of $\Delta$ in $ L^s_n(\hp)$, i.e. $\Delta_{n,\hp}=f^{-1}_{n,\hp}(\Delta)$ where $f_{n,\hp}: L^s_n(\hp) \rightarrow L$. 
By abuse of notation, denote by $f_n^*[\Delta]$ the current $$f_n^*[\Delta]:= \underset{\hp \in \hf^{-n}(\hP)\cap \hP}{\sum} f_{n,\hp}^*[\Delta].$$
By definition of $f_{n,\hp}: L^s_n(\hp) \rightarrow L$, we have $\frac{1}{d^n} f_n^*[\Delta]\leq \frac{1}{d^n}f^{n*}[l]$.
By Proposition \ref{nb de tubes}, $\bigcup_{\hp\in \hP\cap\hf^{-n}(\hP)} \Delta_{n,\hp}$ contains at least  $ d^n\hnu(\hP)(1-\varepsilon)$ disjoint disks, so 
all cluster values of $\frac{1}{d^n} f_n^*[\Delta]$ is a non trivial positive current, uniformly laminar and is smaller than $ T$.
By construction, they are closed in $L$ and subordinate to $W^s_L(\hP)$.
In fact, $\supp(f^*_{n,\hp}[\Delta])=\Delta_{n,\hp} \subset L^s_n(\hp)$ and  $L^s_n(\hp)$ converges exponentially fast to $ W^s_L(\hp)$.

Thereby, all cluster values of $\frac{1}{d^n} f_n^*[\Delta]$ intersect correctly (in $L$) and are bounded from above by $T$, so the supremum of these cluster values, denoted by $T^s_\hP$, is a well defined laminar current subordinate to $W^s_L(\hP)$, see \cite[Lemma 6.12]{BLS}.
For all $\hp\in \hP$ there exist  infinitely many  $n\in \N$ such that $\hf^n(\hp)\in \hP$, so $W^s_L(\hp)$ is included in the support of at least one cluster value of $\frac{1}{d^n} f_n^*[\Delta]$, thus $W^s_L(\hp)\subset \supp(T^s_\hP)$.
\end{proof}

We deduce the following corollary, see \cite[Lemma 8.2.]{BLS}.

\begin{coro}
There exists a continuous psh function $u^s_\hP$ defined on $L$ such that for all cut-off function  $\chi$ with support in $L$ we have $\chi T^s_\hP=\chi \dd^c u^s_\hP$.

The product $ T^s_\hP\intgeom T^u$ is well defined in $L$ and $M(T^s_\hP\intgeom T^u)>0$.
\end{coro}

\subsection{Proof of Theorem \ref{th laminarite sans ens. att.}}

\begin{proof}[Proof of Theorem \ref{th laminarite sans ens. att.}]
Fix a Pesin box $\hP$.
The current $T^s_\hP$  of Theorem \ref{constr. TshP} is uniformly laminar, so $\supp \left(\frac{1}{d^n}f^{n*}( T^s_\hP)\right) = f^{-n}\left(W_L^s(\hP)\right)$.
We know that $\supp(\mu_{eq})\cap W_L^s(\hP)=\emptyset$ and $\supp(\mu_{eq})$ is totally invariant thus
$\supp \left(\frac{1}{d^n}f^{n*}( T^s_\hP)\right)$ does not intersect  $\supp(\mu_{eq})$.

Since $\partial f^{-n}(L)$ is smooth, $T^s_\hP\leq T$, and $T\wedge T=0$ outside $\supp(\mu_{eq})$,
we know, by Proposition \ref{prop intersection nulle}, that $\frac{1}{d^n}f^{n*}( T^s_\hP)$ is still a uniformly laminar current subordinate to  $W^s(\nu):= \bigcup_{x\in \supp(\nu)} W^s(x)$,
and the currents $ T^s_\hP,\cdots,\frac{1}{d^n}f^{n*}( T^s_\hP)$ intersect correctly.
So 
$$ T_n=\max \left\lbrace  T^s_\hP,\cdots,\frac{1}{d^n}f^{n*}( T^s_\hP) \right\rbrace$$ is a well defined uniformly laminar current, see \cite[Lemma 6.11]{BLS},  and  $$\supp(T_n)=\underset{i\in \{ 0,\cdots,n\}}{\bigcup} \supp(f^{i*}( T^s_\hP).$$

For all $n\in\N$, we have  $\frac{1}{d^n}f^{n*}( T^s_\hP)\leq T$ so $(T_n)$ is a non-decreasing sequence of  uniformly  laminar currents bounded by $T$ thus $T^s=\sup_n (T_n)$ is well defined. 
The current $T^s$ is laminar and subordinate to  $W^s(\nu)$, and satisfies $T^s\leq T$.
For all $\hp\in \hP$ there exist  infinitely many  $n\in \N$ such that $\hf^n(\hp)\in \hP$ so $W^s(\hP)\subset \supp (T^s)$.

We may do the same for all Pesin box $\hP$.
Thereby, by taking the supremum on the Pesin boxes,
we obtain a laminar current $T^s\leq T$ such that for  $\nu-$almost every $x$,   $W^s(x)\subset \supp(T^s)$. 
\end{proof}

\section{Proof of Theorem \ref{th laminaire dans bassin}}

We are going to prove Theorem \ref{th laminaire dans bassin}, and its corollary, under weaker assumptions.

 \begin{rmk}
It is not clear that $\supp(\nu)$ being included in an attracting set is enough to ensure that $\nu$ admits a negative Lyapunov exponent.
 \end{rmk}
 \begin{theoreme}\label{th laminaire dans bassin2}
Let $f$ be an endomorphism of $\proj^2$ of degree $d$ and $T$ be its  Green current. If $f$ admits a trapping region $U$, such that the conditions \ref{cond att set}, \ref{cond CV} and \ref{cond on nu} are satisfied, then $T$ is laminar subordinate to the stable manifolds $\bigcup_{x\in \supp(\nu)} W^s(x)$ in the basin of attraction $\Br_\A=\underset{n\geq 0}f^{-n}(U)$.
\end{theoreme}

\begin{coro}\label{coro de la laminarite}
Under the assumptions  \ref{cond att set}, \ref{cond CV} and \ref{cond on nu}, for $\sigma_T$-almost every $p\in \Br_\A$, we have 
$$\dfrac{1}{n} \sum_{i=0}^n \delta_{f^i(p)} \rightharpoonup \nu.$$ 
\end{coro}

We start with the following proposition:

\begin{prop}\label{laminar in the basin}
In the basin of attraction $\Br_\A$ of $\A$, we have
$$\dfrac{1}{d^n}f^{n*}T^s_\hP\rightarrow cT \text{ where } c=M(T^s_\hP\intgeom T^u)>0.$$
\end{prop}

\begin{proof}
Let $\psi$ be a smooth cut-off  function with support in $L$ and $\phi$ a $(1,1)$-smooth form with support in $\Br_\A$, then, by hypothesis \ref{cond CV}, $\dfrac{1}{d^n}f^n_* \phi \rightarrow \langle T,\phi\rangle T^u$.
If the potential $u^s_\hP$ of $T^s_\hP$ is smooth then
\begin{align*}
\langle \frac{1}{d^n}f^{n*}(\psi T^s_\hP),\phi \rangle & = \langle \psi T^s_\hP , \frac{1}{d^n}f^{n}_*(\phi) \rangle \\
& = \int u^s_\hP\left( \dd^c(\psi)\wedge \frac{1}{d^n}f^{n}_*(\phi) + d\psi\wedge \frac{1}{d^n}f^{n}_*(d^c\phi)+\psi \frac{1}{d^n}f^{n}_*(\dd^c \phi)\right).
\end{align*}

Otherwise, since $\dd^c(\psi)$, $\phi$, $d\psi$, $d^c\phi$, $\psi$ and $\dd^c \phi$ are smooth forms and the push forward of a smooth form is a current with continuous coefficients,
each term of the sum $\dd^c(\psi)\wedge \frac{1}{d^n}f^{n}_*(\phi) + d\psi\wedge \frac{1}{d^n}f^{n}_*(d^c\phi)+\psi \frac{1}{d^n}f^{n}_*(\dd^c \phi)$ is well defined and has continuous coefficients.
We can defined $\frac{1}{d^n}f^{n*}(\psi T^s_\hP)$ by $$\langle \frac{1}{d^n}f^{n*}(\psi T^s_\hP),\phi \rangle = \int u^s_\hP\left( \dd^c(\psi)\wedge \frac{1}{d^n}f^{n}_*(\phi) + d\psi\wedge \frac{1}{d^n}f^{n}_*(d^c\phi)+\psi \frac{1}{d^n}f^{n}_*(\dd^c \phi)\right).$$
Moreover, $u^s_\hP$ is continuous so  $\phi \mapsto \langle \frac{1}{d^n}f^{n*}(\psi T^s_\hP),\phi \rangle$ is also  continuous.

We know that $||\frac{1}{d^n}f^{n}_*(d^c\phi)||\rightarrow 0$ and $||\frac{1}{d^n}f^{n}_*(\dd^c \phi)||\rightarrow 0$, see \cite[Proposition 4.7]{DINH}.
We conclude that, for all $(1,1)$-smooth form  $\phi$  with support in  $\Br_\A$, we have $\langle \frac{1}{d^n}f^{n*}(\psi T^s_\hP),\phi \rangle \rightarrow \langle T,\phi \rangle \langle \psi T^s_\hP,T^u  \rangle$.
\end{proof}

We now prove Theorem \ref{th laminaire dans bassin2}.

\begin{proof}[Proof of Theorem \ref{th laminaire dans bassin2}]
As we saw in the proof of Theorem \ref{th laminarite sans ens. att.}, $  T_n=\max \lbrace  T^s_\hP,\cdots,\frac{1}{d^n}f^{n*}( T^s_\hP) \rbrace$ is a well defined uniformly laminar current   subordinate to $W^s(\A)$. 
Moreover, $\sup_n (T_n)$ is a well defined laminar current   subordinate to $W^s(\A)$ and satisfy $\sup_n (T_n)\leq T$.
By Proposition \ref{laminar in the basin}, the non-decreasing  limit of $(T_n)$ is equal to $cT$. 
Thereby, the restriction of the Green current $T$ to the basin of attraction of $\A$ is a laminar current subordinate to $W^s(\A)$.
\end{proof}

We end this section with the proof of Corollary \ref{coro de la laminarite}.
Denote $\Br_\nu$ the basin of attraction of $\nu$, i.e. 
$$\Br_\nu: =\left\lbrace p \, | \, \dfrac{1}{n} \sum_{i=0}^n \delta_{f^i(p)} \rightharpoonup \nu  \right\rbrace.
$$
The proof is based on the fact that if $p\in W^s(q)$ then 
\begin{equation}\label{eq meme limite}
\frac{1}{n} \sum_{i=0}^n \delta_{f^i(p)}- \frac{1}{n} \sum_{i=0}^n \delta_{f^i(q)} \rightharpoonup 0.
\end{equation}
In fact, let $\varphi$ be a continuous function defined on $\Br_\A$ and let $p,q\in \Br_\A$ such that $p\in W^s(q)$, then $\underset{n\rightarrow +\infty}{\lim}  \, \frac{1}{n} \sum_{i=0}^n \left(\varphi\circ f^i (p)- \varphi\circ f^i (q)\right)=0$.

\begin{proof}[Proof of Corollary \ref{coro de la laminarite}]
Fix $\varepsilon>0$ and denote by $A_\varepsilon$ the set $A_\varepsilon=\pi_0(\hR_\varepsilon)\cap \Br_\nu$. 
Let $\hP\subset \hR_\varepsilon$ be a  Pesin box,  $L$ be a common  Lyapunov chart of $\hP$ and $T^s_\hP$ be the current constructed in Theorem \ref{constr. TshP}.

We may define the restriction $T^u_\hP$  of $T^u$ to $W^u(\hP)$. See Section 1.3 and 1.6 of \cite{DDG3} for more details. By \cite[Theorem 1.6]{DDG3}, $T^u_\hP$ is a uniformly woven current.
Denote by  $\nu^s, \nu^u$ the measure such that 
$$T^s_\hP=\int [W^s_\alpha] \ds \nu^s(\alpha), \, T^u_\hP=\int [\Delta^u_\beta] \ds \nu^u(\beta)$$
$$\text{ and } \nu_\hP:=T^s_\hP\intgeom T^u_\hP=\int [W^s_\alpha]\intgeom [\Delta_\beta] \ds \nu^s(\alpha)\otimes \nu^u(\beta).$$
Since $\nu$ is ergodic, by Birkhoff theorem, we have $\nu(\Br_\nu)=1$ so $0= \nu_{|\hP}(A_\varepsilon^c)\geq \nu_\hP(A_\varepsilon^c)$, where $\nu_{|\hP}$ is the restriction of $\nu$ to $\hP$.
By Fubini Theorem, for $\nu^s$-a.e. $\alpha$ we have $$\int [W^s_\alpha]\intgeom [\Delta^u_\beta] (A_\varepsilon^c) \ds \nu^u(\beta)=0.$$
We deduce that for $\nu^s$-a.e.  $\alpha$, we have $$\int [W^s_\alpha]\intgeom [\Delta^u_\beta] (A_\varepsilon) \ds \nu^u(\beta)=\int [W^s_\alpha]\intgeom [\Delta^u_\beta] (\pi_0(\R_\varepsilon)) \ds \nu^u(\beta) \geq\int [W^s_\alpha]\intgeom [\Delta^u_\beta] (\pi_0(\hP)) \ds \nu^u(\beta) >0.$$
In particular, for $\nu^s$-a.e.  $\alpha$, $W^s_\alpha\cap A_\varepsilon\neq \emptyset$.
Since $\sigma_{T^s_\hP}=\int [W^s_\alpha]\intgeom \omega_{FS} \ds \nu^s(\alpha)$, for $\sigma_{T^s_\hP}$-a.e. $p\in\Br_\A$ there exists $\alpha$ such that $p\in W^s_\alpha$ and $W^s_\alpha\cap A_\varepsilon\neq \emptyset$.
Thus there exists $q\in \A_\varepsilon$ such that $p\in W^s_\alpha=W^s_L(q) $.
By \eqref{eq meme limite}, we have $\frac{1}{n} \sum_{i=0}^n \delta_{f^i(p)}\rightharpoonup  \nu$.

Denote $T_n:=\underset{1<i<n}{\max} f^{i*} T^s_\hP$, since $f^{-1}(\Br_\nu)=\Br_\nu$ and $\supp(f^{i*} T^s_\hP)=f^{-i}\supp( T^s_\hP)$, we have $\sigma_{T_n}(\Br_\A\setminus \Br_\nu)=0$.
We deduce from the proof of Theorem \ref{th laminaire dans bassin} that  $\left(\sigma_{T_n} \right)$ is a non-decreasing sequence of measures converging to $\sigma_T$ thus 
 $$\sigma_T(\Br_\A\setminus \Br_\nu)=\lim \, \sigma_{T_n}(\Br_\A\setminus \Br_\nu)=0.$$
Hence for $\sigma_{T}$-a.e.  $p\in\Br_\A$ we have $\frac{1}{n} \sum_{i=0}^n \delta_{f^i(p)}\rightharpoonup \nu$.
\end{proof}

\section{Equidistribution of saddle periodic points}\label{section equidistribution}

In this section, we follow the ideas of \cite{BLS} which have also been used in \cite{DDG3}.
We assume that $f$  admits an attracting set $\A$ which satisfies the condition \ref{cond U rectract} and that there exists an invariant current  $T^u$ ($\frac{1}{d}f_*T^u=T^u$) such that the measure $\nu=T\wedge T^u$ satisfies \ref{cond forte on nu}.

\begin{rmk}
In this section, we do not need to have any convergence toward the current $T^u$.
\end{rmk}

 Denote by $\Br_\A$ the basin of attraction of $\A$.
 Fix a Pesin box $\hP$ and a  Lyapunov chart $L$ of $\hP$.

\begin{lemme}[Shadowing lemma]\label{Shadowing lemma}
For all $\hx\in  \hP\cap \hf^{-n}(\hP)$ there exists a (unique) periodic point  $\kappa(\hx)\in L^s_n(\hx)\cap L^u_n(\hf^n(\hx))$ of period $n$.
\end{lemme}

\begin{proof}
See \cite[p.284]{BLS2} or \cite[Section 9 Step 2]{DDG3}.
\end{proof}

Denote by $Per_n$ the set of periodic points of period $n$ and
$P_{L,n}=\lbrace \kappa(\hp)\, |\,  \hp\in\hP\rbrace$. 
For all $\kappa\in P_{L,n}$ denote
 $$ \Omega(\kappa) = \lbrace \hx \in  \hP\cap \hf^{-n}(\hP) \, |\, \kappa (\hx)= \kappa \rbrace.$$
Let us recall that if $\hp,\hq\in \hP\cap f^{-n}(\hP) $ then $L^s_n(\hp)\cap L^s_n(\hq)$ is empty or  $L^s_n(\hp) = L^s_n(\hq)$, see Lemma \ref{lemme Lsn= compo connex}.
So  
$$ \Omega(\kappa)\subset\hP\cap \hf^{-n}(\hP)  \cap \pi_0^{-1}( L^s_n(\kappa)),$$
where $L^s_n(\kappa)=L^s_n(\hx)$ if $\kappa (\hx)= \kappa$.

\begin{lemme}
We have $\underset{n\rightarrow\infty}{\liminf}\dfrac{1}{d^n}\sharp  (Per_n\cap L)\geq \hnu(\hP).$
\end{lemme}

\begin{proof}
By Lemma \ref{majoration nu_{|L^s_n}}, we know that for every $\kappa \in P_{L,n}$, $\hnu(\Omega(\kappa))\leq d^{-n}\hnu(\hP)$.
By Lemma \ref{Shadowing lemma}, $\hP\cap \hf^{-n}(\hP)$ is the disjoint union $\underset{\kappa\in P_{L,n}}{\bigcup}\Omega(\kappa)$, thus
$$\hnu(\hP\cap \hf^{-n}(\hP))=\sum_{\kappa\in P_{L,n}}\hnu(\Omega(\kappa))\leq  \frac{1}{d^n}\hnu(\hP)\sharp( P_{L,n})$$
and $d^{-n}\sharp( P_{L,n})\geq \frac{\hnu(\hP\cap \hf^{-n}(\hP))}{\hnu(\hP)}$.
Since $P_{L,n}\subset Per_n\cap L$, we have $$\underset{n\rightarrow\infty}{\liminf}\frac{1}{d^n}\sharp  (Per_n\cap L)\geq  \dfrac{\hnu(\hP\cap \hf^{-n}(\hP))}{\hnu(\hP)}.$$
The measure $\nu$ is mixing, we conclude that $\underset{n\rightarrow\infty}{\liminf}\frac{1}{d^n}\sharp  (Per_n\cap L)\geq \hnu(\hP).$ 
\end{proof}

\begin{lemme}\label{equidistrib part sup}
Let $$\nu_n=\dfrac{1}{d^n}\sum_{\kappa \in Per_n\cap U} \delta_\kappa$$ and $\tilde{\nu}$ be a cluster value of $(\nu_n)$, then $\tilde{\nu}\geq \nu$.
\end{lemme}

\begin{proof}
Since every set can be cover, up to a $\nu$-null set, by Pesin boxes and $\supp(\nu)\subset \Br_\A$, the result follows from the previous Lemma.
\end{proof}

\begin{theoreme}\label{th equi pt per}
Let $f$ be an endomorphism of $\proj^2$ which admits an attracting set $\A$ which $f$  admits an attracting set $\A$. Assume, moreover, $\A$ admits a trapping region satisfying the conditions \ref{cond U rectract}, and that there exists an invariant current  $T^u$ ($\frac{1}{d}f_*T^u=T^u$) with support on $\A$ such that the measure $\nu=T\wedge T^u$ satisfies \ref{cond forte on nu}.   
Then $$\nu_n=\dfrac{1}{d^n}\sum_{\kappa \in Per_n\cap \Br_\A} \delta_\kappa \rightarrow \nu.$$
\end{theoreme}

 \begin{proof}
The restriction of $f$ to an  invariant curve is of topological entropy $\log(d)>0$ so $f$ cannot have a curve of  fixed  points.
Since $f(U)\Subset U$, $\A=\underset{n\in\N}{\bigcap}f^n(U)$ and $\Br_\A=\underset{n\in\N}{f^{-n}(U)}$, there is no fixed points in $\Br_\A\setminus U$.
The compact set $\overline{U}$ is an euclidean retract, see \cite[Proposition/Definition IV 8.5]{dol}, and $f(\overline{U})\Subset \overline{U}$ is compact.
Hence by  Lefschetz-Hopf theorem, see \cite[Proposition VII 6.5]{dol}, the number of  periodic  points of period $n$ in 
 $\overline{U}$ is $\sum \text{Trace}\left((f^n)^*_{|H^i(\overline{U},\Q)}\right)$.
 
Since $\overline{U}$ retracts on $\ell$,  $H^i(\overline{U},\Q))=H^i(\ell,\Q))=H^i(\proj^1\C,\Q))$ and for a generic line $\Delta\subset U$ we have  $f^n_*\Delta \cdot \Delta=d^n$, thus $ \sharp (Per_n\cap U) =d^n +1$.
Thereby, every cluster value of $\nu_n=\dfrac{1}{d^n}\sum_{\kappa \in Per_n\cap U} \delta_\kappa$ has mass 1 and we conclude with Lemma \ref{equidistrib part sup}.
 \end{proof}

 
 \section{Uniqueness of the measure of maximal entropy}\label{section uniqueness}

We assume that $f$ is an endomorphism of $\proj^2$ admitting a non trivial attracting set $\A$ and that conditions  \ref{cond unicite}, \ref{cond on nu} and \ref{cond all hyp} are satisfied

We still denote by $\Br_\A$ the basin of attraction of $\A$, $T$ the Green current of $f$ and $T^u$ the attracting current. The aim of this section is to prove the following Theorem:

\begin{theoreme}\label{th unicité mesure entropie max}
Let $f$ be an endomorphism of $\proj^2$ admitting a non trivial attracting set $\A$  which satisfies the condition \ref{cond unicite}, \ref{cond on nu} and \ref{cond all hyp}  then $\nu=T\wedge T^u$ is the unique measure of maximal entropy $\log(d)$ in $\Br_\A$.
\end{theoreme}

To prove this theorem we follow the approach of \cite{BLS}.\\

\begin{proof}[Proof of Theorem \ref{th unicité mesure entropie max}]
Let $\m$ be a $f$-invariant measure with support in $\Br_\A$ and of (maximal) entropy $\log(d)$ then, by Choquet representation theorem, $\nu$ can be written as an integral of ergodic measures which also are of maximal entropy $\log(d)$, since the metrical entropy is concave. So we only have to prove that $\nu$ is the only ergodic measure of maximal entropy in $\Br_\A$.
Let $\m$ be an ergodic measure of maximal entropy with support in $\Br_\A$.
By \ref{cond all hyp}, $\m$  admits a non positive Lyapunov exponent.

The measure $\hm$ is also ergodic and, by Birkoff theorem, for every continuous function $\varphi$ and $\hm$-a.e.  $\hx$ we have 
\begin{equation}\label{utilisation de birkoff}
\underset{n\rightarrow \infty}{\lim} \dfrac{1}{n}\sum_{i=0}^{n-1}\varphi(\hf^n(\hx))=\int \varphi(\hx)\ds \hm(\hx).
\end{equation}
Denote $\hm_\hx:= \hm(\cdot|\hxi^u)$ then, by dominate  convergence and \eqref{utilisation de birkoff}, we have 
\begin{eqnarray*}
\underset{n\rightarrow \infty}{\lim} \int_{\hxi^u(\hx)} \varphi(\hy) \ds \left(\dfrac{1}{n}\sum_{i=0}^{n-1} \hf^n_*(\hm_\hx)\right)(\hy) 
& =& \underset{n\rightarrow \infty}{\lim} \int_{\hxi^u(\hx)} \dfrac{1}{n}\sum_{i=0}^{n-1}\varphi(\hf^n(\hy)) \ds \hm_\hx(\hy)\\
& = &\int_{\hxi^u(\hx)} \left(\underset{n\rightarrow \infty}{\lim} \dfrac{1}{n}\sum_{i=0}^{n-1}\varphi(\hf^n(\hy)) \right)\ds \hm_\hx (\hy) \\
& = & \int \varphi(\hy)\ds \hm(\hy)
\end{eqnarray*}
hence 
\begin{equation}\label{cv vers hm}
\dfrac{1}{n}\sum_{i=0}^{n-1} \hf^n_*(\hm_\hx)\rightharpoonup \hm.
\end{equation}

One the other hand, for every $\hx$, since $$M\left(\dd^c \left(\dfrac{1}{d^n}f^n_*[D^u(\hx)]\right)\right)= M\left(\dfrac{1}{d^n}f^n_*\dd^c \left([D^u(\hx)]\right)\right)=O\left(\dfrac{1}{d^n}\right),$$
every cluster value of  $\left(\frac{1}{d^n}f^n_*[D^u(\hx)]\right)$ is a positive closed current of support $M([D^u(\hx)])$ with support in $\A$, 
see \cite{DINH} for more details.
By the condition \ref{cond unicite}, we have
$$\frac{1}{d^n}f^n_*[D^u(\hx)]\rightarrow c \, T^u$$ with $c=M([D^u(\hx)])$.
The Green current $T$ has a continuous potential so 
\begin{equation*}
\dfrac{1}{\rho(\hx)} T\intgeom\left(\frac{1}{d^n}f^n_*[D^u(\hx)]\right)\rightarrow \frac{c}{\rho(\hx)}  T\intgeom  T^u=\frac{c}{\rho(\hx)}  \nu.
\end{equation*}
So $c=\rho(\hx)$, since $\dfrac{1}{\rho(\hx)} T\intgeom[D^u(\hx)]$ and $\nu$ are probability measures.
Thus we have 
\begin{equation}\label{cv vers nu}
f^n_*\left(\dfrac{T\intgeom[D^u(\hx)]}{\rho(\hx)} \right)\rightarrow \nu.
\end{equation}

Thanks to Proposition \ref{prop conditionnelle induite par green}, we know that $\hm_\hx=(\pi_0^{-1})_*\left(\frac{1}{\rho(\hx)}T\intgeom[D^u(\hx)]\right)$ but \eqref{cv vers nu} is not enough to conclude that 
$$\hf^n_*(\hm_\hx)=f^n_*\left((\pi_0^{-1})_*\dfrac{T\intgeom[D^u(\hx)]}{\rho(\hx)} \right)\rightarrow \hnu.$$
However, for all $\hX\subset\hA$ we have
\begin{eqnarray*}
\hf^n_*(\hm_\hx)(\hX) & = & \dfrac{1}{\rho(\hx)}T\intgeom [D^u(\hx)]\left(\pi_0(\hf^{-n}(\hX))\right)\\ 
& = &\dfrac{1}{\rho(\hx)} T\intgeom [D^u(\hx)]\left(\pi_{-n}(\hX)\right).
\end{eqnarray*}
If $0\leq i \leq n$ then $\pi_{-n}(\hX)\subset f^{-(n-i)}\pi_{-i}(\hX) $ so
\begin{eqnarray*}
\hf^n_*(\hm_\hx)(\hX) & \leq &\dfrac{1}{\rho(\hx)} T\intgeom [D^u(\hx)]\left(f^{-(n-i)}\pi_{-i}(\hX))\right)\\
& \leq & \dfrac{1}{\rho(\hx)} T\intgeom \dfrac{1}{d^{-(n-i)}} f^{-(n-i)}_*[D^u(\hx)]\left(\pi_{-i}(\hX)\right).
\end{eqnarray*}
We let $n$ go to the infinity, and, by  \eqref{cv vers hm} and \eqref{cv vers nu}, we obtain that for every $i\in\N$
$$\hm(\hX) \leq \nu\left(\pi_{-i}(\hX)\right),$$
so $\hm\leq \hnu$.
But $\hm$ and $\hnu$ are probability measures so $\hm=\hnu$, and $\m=\nu$.
This end the proof of Theorem \ref{th unicité mesure entropie max}.
\end{proof}

\section{Further remarks}\label{section rmk}
\subsection{Hypothesis for Theorem \ref{th laminaire dans bassin}}\label{section hyp th lam bassin}

\subsubsection{Previously known settings}

The conditions given in the introduction may do reformulate thanks to \cite{moi:taf,DINH,taf}.
In fact, in \cite{DINH}, T.C. Dinh prove that if
$U$ contains an image of  $\proj^1(\C)$  and $\proj^2\setminus U$ is start-shapped, then $U$ support a natural positive closed current $T^u$ of bidegree $(1,1)$, and $\nu=T\wedge T^u$ is mixing of entropy $\log(d)$.

  \begin{rmk}
It is not clear that the fact that $\supp(\nu)$ is include in an attracting set is enough to ensure that $\nu$ admits a negative Lyapunov exponent.
 \end{rmk}

If we further assume that $f$ is of small topological degree on $U$  then $f$ satisfies \ref{cond unicite} and \ref{cond forte on nu}, see \cite{moi:taf}. This two conditions are also true in the setting of \cite{taf}, i.e. if  for all $x\in U, \,  ||D_x f||<1 $  and there exist a point $ I\notin U$ and a line $ \ell\subset U$  such that  for all $x\in \ell $ the set $I(x)\cap U\subset I(x)\setminus I\simeq \C^2$ is strictly  convex, where $ I(x)$  is the line passing through $ I$  and $x.$

 If instead we further assume that the rational hull $r(K)$ of the compact set $K= \proj^2\setminus U$ (see \cite[Definition 2.1]{guedj}) does not intersect $\A$, then $f$ satisfies \ref{cond CV}.
This follows from \cite[Lemma 2.7]{guedj}  and \cite[Theorem 4.6]{DINH}.

\subsubsection{New examples}

In practice, the only examples known that satisfy \ref{cond CV} (and also all the assumptions in section \ref{section assumptions}) are the perturbations of the line at infinity exposed in the introduction.
 In this section, we present new examples.
 
We fix the following notations:
 $$F_\theta:[x:y:z]\mapsto [x^2:y^2:xy+\theta(z^2-xy)]$$
$$U=\{[x:y:z]\, | \, |z^2-xy|\leq \delta \max(|x|,|y|,|z|)^2\}$$
and $X=x^2,Y=y^2, Z=xy+\theta(z^2-xy) $.
Assume that $0<|\theta |\leq \theta_0$ and that $\theta_0$ and $\delta$ are small, then $[0:0:1]\notin U$ and
for all $[x:y:z]\in U$, $\theta_0|z^2|\leq \max (|x|^2,|y|^2)\leq \max(|X|,|Y|,|Z|)$. Thus for $[x:y:z]\in U$ we have $\theta_0\max(|x|,|y|,|z|)^2\leq  \max(|X|,|Y|,|Z|)$ and
\begin{eqnarray}
|Z^2-XY|&=& |\theta (z^2-xy) | |xy-Z|\\ \nonumber
&\leq & \delta \theta_0 \max(|x|,|y|,|z|)\cdot 2\max(|X|,|Y|,|Z|)\\ \nonumber
&\leq & 2\delta \max(|X|,|Y|,|Z|)^2 \nonumber
\end{eqnarray}
So $F_\theta(U)\subset \subset U$ and $\underset{n\geq 0}{\bigcap}F_\theta^n(U)=\{z^2=xy\}$, i.e. the conic $\{z^2=xy\}$ is an attracting set for $F_\theta$.

\medskip

The trapping region $U$ does not satisfy the hypotheses of Dinh \cite{DINH} but we are going to prove :

\begin{prop}
If $f$ is an endormorphism of $\proj^2$ such that $U=\{[x:y:z]\, | \, |z^2-xy|\leq \delta \max(|x|,|y|,|z|)^2\}$ is a trapping region for $f$ then $f$ satisfy \ref{cond CV}.
\end{prop}

\begin{rmk}
Even if it seems possible, it is not clear how to exhibit an example such that no trapping region satisfy Dinh's hypotheses.
We may consider the family of maps
$$\tilde{F}_\theta[x:y:z]\mapsto [P(x,y,z)^2:Q(x,y,z)^2:P(x,y,z)Q(x,y,z)+\theta(z^2-xy)^d],$$
where  $P,Q$ are homogeneous polynomials of $\C[X,Y,Z]$ of degree $d\geq 1$ such that, for all $\theta\neq 0$, $\tilde{F}_\theta$ is an endomorphism of $\proj^2$, and the indeterminacy points of $\tilde{F}_0$ are not on$\{z^2=xy\}$.
\end{rmk}

\begin{lemma}
The rational hull $r(K)$ of the compact set $K= \proj^2\setminus U$ (see \cite[Definition 2.1]{guedj}) is equal to $K$.
\end{lemma}

\begin{proof}
Let $[a_1:a_2:a_3]\in U$ and chose $i$ such that $|a_i|= \max(|a_1|,|a_2|,|a_3|)$, so $|a_1 a_2-a_3^2|<\varepsilon |a_i|^2$. Denote by $\widetilde{\Cc}$ the conic $\widetilde{\Cc}=\{[x_1:x_2:x_3]\, :\, a_i^2(x_1 x_2-x_3^2)=x_i^2(a_1 a_2-a_3^2)\}$. Let $[x_1:x_2:x_3]\in \widetilde{\Cc}$ then
$$|x_1 x_2-x_3^2|=\dfrac{|a_1 a_2-a_3^2|}{|a_i^2|}|x_i|^2 < \varepsilon \max(|x_1|,|x_2|,|x_3|).$$
Thus the conic $\widetilde{\Cc}$ is included in $U$ and contains $[a_1:a_2:a_3]$.
\end{proof}

Thanks to \cite[Lemma 2.7]{guedj}, we only have to see how to adapt Dinh's proofs to get the conclusion of Theorem 1.1 and  Theorem 4.6 of \cite{DINH}.
Let $R$ be a positive closed current with continuous coefficients and support in $U$. 
The only time Dinh uses the assumption on the geometry of $U$ is in the section 3 to construct the structural disks.
Here is how we can do it in this situation.

Fix a chart $W$ of $Aut(P^2)$ containing $\id$ and local holomorphic coordinates $A$, such that $||A|| < 1$ and  $A = 0$ at id.
Denote by :
\begin{itemize}
\item $W' \Subset W$ a small neighbourhood of $\id$,
\item $U'$ an open set such that $f(U)\Subset U'\Subset U$,
\item $V$ a simply connected neighbourhood of the interval $[0,1]$ in $\C$,
\item for $\alpha \in \C$ and $||A||\leq \min(1,|1-\alpha|^{-1})$, $\lambda_\alpha (A):= (1-\alpha)A$,
\item $\pi_1, \pi_2$ the projections of $V\times\proj^2$ to the first and second coordinates.
\end{itemize}

We choose $V,W',\theta$ small enough such that for all $\alpha \in V$, $A\in W'$, and all $p\in F_\theta^{-1}(U')$ we have $\lambda_\alpha(A)\circ F_{\alpha\theta} (p)\in U$.
For all $A\in W'$, denote by:
\begin{eqnarray}\nonumber
\Fc_A : V\times\proj^2 &\rightarrow & V\times \proj^2\\ \nonumber
 (\alpha, p)&\mapsto & (\alpha, \lambda_\alpha(A)\circ F_{\alpha\theta}(p)) 
\end{eqnarray}
This is a holomorphic endomorphism outside $\{0\}\times \proj^2$. Since $\Fc(\{0\}\times \proj^2)=\{0\}\times A(\Cr)$, we can extend trivially the current
$$\Rr_A=\dfrac{1}{d}\Fc_{A*}\left(\pi_2^*\left(\dfrac{1}{d} F_\theta ^*R\right)\right)$$
to $V\times\proj^2$ as a positive closed current, see \cite{har:pol}.
We define $R_{\alpha,A}$ to be the structural discs $R_{\alpha,A}=<\Rr_A,\pi_2,\alpha>$, see \cite[Appendix A]{DINH} for notations.
Let $\rho$ be a smooth positive probability measure with compact support in $V$ and denote 
$$R_\alpha=\int R_{\alpha,A} \, \ds \rho (A).$$
We have $R_1=\dfrac{1}{d} F_{\theta*}\left(\dfrac{1}{d} F_\theta ^*R\right)=R$, since $\lambda_1(A)=\id_{\proj^2}$, and $R_0=\int A_*[\Cr] \, \ds \rho (A)$ is independent of $R$.

\medskip

The end of the proof is exactly the same than the one of Dinh.

\subsection{Around Theorem \ref{th laminarite sans ens. att.} } 

Let $\nu$ be as in theorem \ref{th laminarite sans ens. att.}, i.e. 
\begin{itemize}
\item $\nu$ is of the form invariant current $\nu=T\wedge T^u$, where $T^u$ is an invariant current ($\frac{1}{d}f_*T^u=T^u$) and $T$ is the Green current of $f$, 
\item $\nu$ is of entropy $\log(d)$ and hyperbolic of saddle type,
\item $\supp(\nu)\cap \supp(\mu_{eq})=\emptyset$, where $\mu_{eq}=T\wedge T$ is the equilibrium measure.
\end{itemize} 

Denote by $\Br_\nu$ the basin of $\nu$. It follows from Theorem \ref{th laminarite sans ens. att.} that $\sigma_T(\Br_\nu)>0$. A natural question is to know if $\sigma_T$ almost every point is in the basin of a hyperbolic measure of saddle type.
\bigskip

For every  Pesin box $\hP$ we constructed (see Theorem \ref{constr. TshP}) a laminar current $T^s_\hP$ such that $ T^s_\hP \intgeom T^u \leq T\wedge T^u=\nu$ and $M(T^s_\hP \intgeom T^u)\geq \nu(\hP)(1-\varepsilon)\cdot\nu^u(\hP)$
where $\nu^u$ is the marking of the restriction $T^u_\hP$  of $T^u$ to $W^u(\hP)$, i.e. the measure such that
$T^u_\hP=\int [\Delta^u_\beta] \ds \nu^u(\beta)$.

For different reasons we believe that Theorem \ref{constr. TshP} can be  improve. 
\begin{question}
Under the same assumptions, can we construct, and  all $\varepsilon>0$, a uniformly laminar current $T^s_\hP$ such that $M(T^s_\hP \intgeom T^u)\geq \nu(\hP)(1-\varepsilon)$.
Or can we construct, for every all $\varepsilon>0$,  a laminar current $T^s_\varepsilon$ such that  $T^s_\varepsilon\leq T$, and
$$M(T^s_\varepsilon\intgeom T^u)\geq 1-\varepsilon \, ?$$
\end{question}

\subsection{The control of the genus of the curves $f^{-n} L$}
 
We mentioned in the introduction that a way to prove that the Green current is laminar is to  control the genus of the curves $f^{-n} L$, where $L$ is a line such that $\frac{1}{d^n}f^{n*}[L]\rightarrow T$.
In several cases, the growth of the genus of the curves $f^{-n} L$ is linked with the number of preimages of a point; the link is given by the Riemann-Hurwitz formula.

More precisely, by \cite[Theorem 1]{deT:laminarite}, if $\genus(f^{-n}(L)\cap U)=O(d^n)$ then $T$ is laminar.
H. de Thelin \cite{deT:concentration} proved that this true for post-critically finite maps. In the general case, he obtained:

\begin{theoreme}[{\cite[Theorem 2]{deT:concentration}}]
For a generic endomorphism $f$ of $\proj^2$, there exists a neighboorhood $V$ of $\mu_{eq}$ such that 
$$\underset{n\rightarrow \infty}{\limsup} \, \dfrac{1}{n} \log \, \left( \underset{L\in (\proj^2)^*}{\max} \, \genus(f^{-n}(L)\cap V^c) \right)\leq \log \, d.$$
\end{theoreme}

By \cite[Theorem 1.1]{duj:contre}, we know that we cannot improve this result without an additional assumption.
An idea is to adapt de Thélin proof in the basin of a small topological degree attracting set (see \cite{moi,moi:taf} for the definition).
The proof of {\cite[Theorem 2]{deT:concentration}} is essentially in two steps:
\begin{enumerate}
\item Controlling the number of ``small'' handles of $f^{-n}(L)$, which is about the same as the number of preimages staying in $U=\proj^2\setminus V$. 
\item Controlling the number of ``larger'' handles. 
\end{enumerate}

Under the small topological degree assumption, we may adapt de Thelin's proof to get that the number of ``small'' handles  is bounded by $O(d^n)$.
But the control of the number of ``larger'' handles does not seems to be linked with the number of preimages, so we only have
that the genus of $f^{-n}(L)$ in a trapping region growth as $O( n \, d^n)$.

Sandrine Daurat

University of Michigan

Ann Arbor, MI, 48109, USA

E-mail:daurat@umich.edu
\end{document}